\documentclass[10pt]{amsart}
\usepackage[utf8]{inputenc}
\usepackage{amsfonts,amssymb}
\usepackage{url}
\numberwithin{equation}{section}
\usepackage{amsthm}
\theoremstyle{plain}
\newtheorem{theorem}{Theorem}[section]
\newtheorem{lemma}[theorem]{Lemma}
\newtheorem{question}[theorem]{Question}
\newtheorem{corollary}[theorem]{Corollary}
\newtheorem*{corollary*}{Corollary}

\newtheorem{proposition}[theorem]{Proposition}

\theoremstyle{definition}
\newtheorem{definition}[theorem]{Definition}
\newtheorem*{definition*}{Definition}
\newtheorem{example}[theorem]{Example}

\theoremstyle{remark}
\newtheorem{remark}[theorem]{Remark}

\newcommand{\C}{\mathbf{C}}
\newcommand{\Z}{\mathbf{Z}}
\newcommand{\cX}{\mathcal{X}}
\newcommand{\cM}{\mathcal{M}}
\newcommand{\cN}{\mathcal{N}}
\newcommand{\E}{\mathbf{E}}
\newcommand{\vn}{\mathcal{C}}
\newcommand{\cH}{\mathcal{H}}
\newcommand{\cU}{\mathcal{U}}
\DeclareMathOperator{\tr}{tr}
\newcommand{\prob}{\mathbf{P}}
\newcommand{\F}{\mathbf{F}}
\DeclareMathOperator{\val}{val}
\DeclareMathOperator{\Tr}{Tr}
\newcommand{\anticom}{\mathrm{anticom}}
\begin{document}

\title{Spectral gap and stability for groups and non-local games}
\date{\today}

\author{Mikael De La Salle}
\address{Université de Lyon, CNRS, France}
\thanks{MdlS was funded by the ANR grants AGIRA ANR-16-CE40-0022 and Noncommutative analysis on groups and quantum groups ANR-19-CE40-0002-01}

\email{delasalle@math.univ-lyon1.fr}
\begin{abstract}
The word stable is used to describe a situation when mathematical objects that almost satisfy an equation are close to objects satisfying it exactly. We study operator-algebraic forms of stability for unitary representations of groups and quantum synchronous strategies for non-local games. We observe in particular that simple spectral gap estimates can lead to strong quantitative forms of stability. For example, we prove that the direct product of two (flexibly) Hilbert-Schmidt stable groups is again (flexibly) Hilbert-Schmidt stable, provided that one of them has Kazhdan's property (T). We also provide a simple form and simple analysis of a non-local game with few questions, with the property that synchronous strategies with large value are close to perfect strategies involving large Pauli matrices. This simplifies one of the steps (the question reduction) in the recent announced resolution of Connes' embedding problem by Ji, Natarajan, Vidick, Wright and Yuen.
\end{abstract}
\maketitle

The aim of this note is to present, in Theorem~\ref{thm:mainTheorem}, the construction of a $2$-player non-local game with $O(N^2)$ questions and $2^N$ answers where any quantum synchronous strategy with value $1-\varepsilon$ is $O(\varepsilon)$-close to a synchronous startegy involving Pauli matrices of size $2^N$ (see Section~\ref{section:games} for the precise definitions). This can be used to simplify the analysis of the Pauli basis test from \cite{MIPRE}, generalize it and provide better quantitative bounds. The proof does not use a specific quantum soundness property of any specific code, and the result is the particular case (applied to asymptotically good codes) of a general construction that takes as input any linear error-correcting code, see Example~\ref{ex:game_from_code}. The main originality in this work is Theorem~\ref{thm:almost_commutation}, which is a very easy consequence of a simple spectral gap argument (Lemma~\ref{lemma:spectral_gap_commutator}). As in previous work of Vidick \cite{vidickPauliBraiding}, the proof also relies on a general form of stability in average for representations of finite groups, but there is a new input in Lemma~\ref{lem:GowersHatami_subgroup}.

It turns out that the same idea has a consequence about Hilbert-Schmidt stability that was apparently not known, although similar ideas are also present in the work of Ioana \cite{MR4134896}. See \S~\ref{sec:stability} for the terminology and a more precise and quantitative statement, that is also relevant for finite groups.
\begin{theorem}\label{thm:stability_direct_product_simple} Let $G$ and $H$ be two countable groups, one of which has property (T).

  If $G$ and $H$ are both Hilbert-Schmidt stable, then so is $G\times H$.

  If $G$ and $H$ are both Hilbert-Schmidt flexibly stable, then so is $G\times H$.
\end{theorem}
The property (T) assumption is crucial, as Ioana recently proved \cite{Ioana2} that the direct product of two finitely generated non-abelian free groups is not flexibly Hilbert-Schmidt stable, whereas free groups are obviously Hilbert-Schmidt stable.
  
In this note, we start from scratch and prove every statement that we need (we sometimes refer to \cite{orthonormalisation}, but the results that rely on it are not used in the proof of the main Theorem). So a significant part of the note consists of known facts, either classical or borrowed from \cite{MIPRE}.

The paper is organized as follows. In Section~\ref{sec:preliminaries}, we present some preliminary facts on Fourier transform for abelian groups, spectral gaps and flexible stability for finite groups. In particular, we prove some variants of results of Gowers and Hatami \cite{MR3733361}. In Section~\ref{sec:main} we prove the main result in terms of stability for direct products of finite groups. In Section~\ref{section:games}, we translate this result in terms of non-local games. Section~\ref{sec:stability}, completely independent from the rest of the paper (except for the use of Lemma~\ref{lemma:spectral_gap_commutator}), is devoted to the proof of Theorem~\ref{thm:stability_direct_product_simple}. 

\subsection*{Acknowledgements} I thank Adrian Ioana, Thomas Vidick and Henry Yuen for useful discussions and comments on a preliminary form of this note. I also thank the anonymous referee for many constructive comment, and for suggesting Remark~\ref{rem:added_by_referee}.
\section{Preliminaries}\label{sec:preliminaries}
\subsection{Matrix and von Neumann algebras notation}
We recall a few basic facts about von Neumann algebras. For proofs and much more, we refer the reader to standard texts such as \cite{MR1873025}, or the modern book project \cite{AnantharamanPopa}.

A von Neumann algebra is a self-adjoint subalgebra of the algebra $B(\cH)$ of bounded operators on a complex Hilbert space $\cH$ that is equal to its bicommutant. Here the commutant of a subset $A\subset B(\cH)$ is the algebra
\[A':=\{X\in B(\cH)\mid  \forall Y \in A,\  XY=YX\},\] and its bicommutant is the commutant of its commutant. Every von Neumann algebra is uniquely a dual space \cite[Corollary III.3.9]{MR1873025}, which allows us to talk about its weak-* topology.

A state on a von Neumann algebra $\cM$ is a linear form $\cM \to \C$ that is positive, meaning that $\tau(x^*x)\geq 0$ for every $x\in \cM$, and normalized by $\tau(1)=1$. It is called faithful if $\tau(x^*x)=0$ holds only if $x=0$. It is called normal if it is continuous for the weak-* topology, and it is called a trace if $\tau(ab)=\tau(ba)$.

In the whole note, we denote by $(\cM,\tau)$ a von Neumann algebra with a normal faithful tracial state. The main case of interest in $\cM = M_n(\C)$ for large $n$ and $\tau$ is the normalized trace $\tr = \frac{1}{n}\Tr$. The reader can safely assume throughout the paper that we are in this situation.

We will denote $\|x\|_2 = \big(\tau(x^*x)\big)^{\frac 1 2}$ for the $L_2$ norm on $\cM$. The completion of $\cM$ for this norm is denoted $L_2(\cM,\tau)$. We also use the notation $\|x\|_q = \big(\tau( (x^*x)^{q/2})\big)^{\frac 1 q}$ for the $L_q$ norm and $1\leq q<\infty$, and $\|x\|_\infty$ for the operator norm. 

If $(\cM,\tau)$ is a von Neumann algebra with a faithful normal tracial state and if $\cN \subset\cM$ is an inclusion of von Neumann algebras, we obtain that $L_2(\cN,\tau)\subset L_2(\cM,\tau)$. The orthogonal projection $L_2(\cM,\tau) \to L_2(\cN,\tau)$ turns out to map $\cM$ into $\cN$, and the corresponding map is called the conditional expectation from $\cM$ to $\cN$ \cite[\S 9]{AnantharamanPopa} or \cite[Proposition V.2.36]{MR1873025}.

We also denote $\cM_\infty = \cM \overline{\otimes} B(\ell_2)$ with (infinite) trace $\tau_\infty:=\tau \otimes \Tr$. We often identify $\cM$ with $\cM \otimes e_{1,1} \subset \cM_\infty$ (with a different unit $1_\cM = 1 \otimes e_{1,1}$ than $1_{\cM_\infty}$), and often use the same letter $\tau$ to denote the amplified trace $\tau_\infty$. We also use the notation $\|x\|_2 = \big(\tau_\infty(x^*x)\big)^{\frac 1 2}$ for $x \in \cM_\infty$, but in this situation the norm can take the value $\infty$.

A PVM (positive valued measure), on a Hilbert space $\cH$ and indexed by a finite set $I$ is a family $(P_i)_{i\in I}$ of self-adjoint projections such that $\sum_i P_i=1$. We talk about PVMs in $\cM$ if $\cM\subset B(\cH)$ and $(P_i)_{i \in I}$ is a PVM on $\cH$ with $P_i \in \cM$ for every $i$.

\subsection{Reminders on Fourier transform}
Groups appearing in this paper will be denoted multiplicatively. We will often denote by $G$ arbitrary groups and by $A$ abelian groups, except in the last section where the letter $A$ will be reserved to answers. If a group $G$ is finite, we will denote $\prob_G$ the uniform probability measure on $G$ and $\E_g f(g)$ the integration with respect to it. When several groups enter the picture and precisions are needed, we sometimes write $\E_{g \in G} f(g)$. We will write $L_2(G)$ and $\ell_2(G)$ for the functions $G \to \C$ with norms
\[\|f\|_{L_2} = \left(\E_g |f(g)|^2\right)^{\frac 1 2}\textrm{ and }\|f\|_{\ell_2} = \left(\sum_{g \in G} |f(g)|^2\right)^{\frac 1 2}\]
respectively. Of course, $L_2(G)$ will only make sense for finite groups.

Let $A$ be a finite abelian group. Recall that a character on $A$ is a group homomorphism $A \to \{z \in \C \mid |z|=1\}$. The set of all characters of $A$, denoted $\hat A$, is a finite group for the operation of pointwise multiplication called the (Pontryagin) dual of $A$. Moreover, the dual of the dual of $A$ identifies naturally with $A$.

The Fourier transform, which implements an isometry between $L_2(A)$ and $\ell_2(\hat A)$, also implements a correspondence between unitary representations of $A$ and PVMs on $\hat A$.
\begin{lemma}\label{lem:Fourier_abelian_group} For every Hilbert space $\cH$, the following maps, inverse of each other, are bijections between unitary representations of $A$ on $\cH$ and PVMs on $\cH$ indexed by $\hat A$~:
  \[ U \mapsto \left(P_\chi = \E_a \overline{\chi}(a) U(a)\right)_{\chi \in \hat A},\]
  \[ (P_\chi)_{\chi \in \hat A} \mapsto \left( U: a \mapsto \sum_\chi \chi(a) P_\chi\right).\]
  \end{lemma}
\begin{proof} This is well-known and follows from the orthogonality of characters. 

\end{proof}

\subsection{Spectral gap preliminaries}\label{sec:spectral_gap}
Let $G$ be a countable group (for example a finite group). If $\mu$ is a symmetric probability measure on $G$ with generating support, then for every unitary representation $(\pi,\cH)$, $\pi(\mu):=\sum_g \mu(g)\pi(g)$ is a self-adjoint operator of norm $\leq 1$, whose eigenspace for the eigenvalue $1$ is the space of invariant vectors $\cH^\pi=\{\xi\in\cH \mid \forall g\in G, \pi(g)\xi=\xi\}$ (this uses that the support of $\mu$ generates the whole group $G$). We say that $\mu$ has spectral gap is there is a positive real number $\kappa$ such that for every unitary representation $(\pi,\cH)$ of $G$, the spectrum of $\pi(\mu)$ in contained in $[-1,1-\frac{1}{\kappa}]\cup\{1\}$, or equivalently is the restriction of $\pi(\mu)$ to the orthogonal of $\cH^\pi$ has spectrum contained in $[-1,1-\frac{1}{\kappa}]$. The smallest such $\kappa$ is denoted by $\kappa(\mu)$, it is the inverse of the spectral gap. If $\mu$ does not have spectral gap, we set $\kappa(\mu)=\infty$. The reason for this parametrization of the spectral gap is because it is proportional to the constant appearing in the Poincar\'e inequalities, and is related to the Kazhdan constants. Specifically, $\kappa(\mu)$ is the smallest real number such that, for every unitary representation $(\pi,\cH)$ of $G$, and every vector $\xi \in \cH$,
\begin{equation}\label{eq:Poincare_constant} \|\xi - P_{\cH^\pi} \xi\|^2 \leq \frac\kappa 2 \int_G \|\pi(g) \xi-\xi\|_2^2d\mu(g).
\end{equation}
Here $P_{\cH^\pi}$ is the orthogonal projection on $\cH^\pi$. Let us justify this characterization of $\kappa$. Indeed, by expanding the square norms, \eqref{eq:Poincare_constant} is equivalent to the inequality
\[ \|\xi\|^2 \leq \kappa (\|\xi\|^2 - \langle\pi(\mu)\xi,\xi\rangle) \forall \xi \in (\cH^\pi)^\perp,\]
or in other words (Rayleigh quotients) the spectrum of $\pi(\mu)$ restricted to $(\cH^\pi)^\perp = \ker(\pi(\mu)-1)^\perp$ consists of real numbers $\lambda$ satisfying $1\leq \kappa(1-\lambda)$, that is $\lambda \leq 1-\frac{1}{\kappa}$.

For a given group $G$, $\mu$ has spectral gap if and only if $G$ has Kazhdan's property (T). In particular, the property that $\kappa(\mu)$ is finite does not depend on $\mu$. But it will be useful, in particular for finite groups, to study measures with good spectral gap. The extreme case is when $\mu =\prob_G$ is the uniform probability on $G$. In that case, $\pi(\prob_G)$ is the orthogonal projection on $\cH^\pi$ and therefore $\kappa(\mu) =1$ (if $G$ is not the group $\{0\}$).

If $\mu$ is a not-necessarily symmetric probability measure on $G$ with support still generating, we define $\kappa(\mu)$ as $\kappa(\nu)$ where $\nu$ is the symmetric measure $\nu(g) = \frac{1}{2}(\mu(g)+\mu(g^{-1})$. It is also the smallest constant such that \eqref{eq:Poincare_constant} holds.

\subsection{Probability measures with spectral gap and small support}
If $A$ is a finite abelian group and $\mu$ is a probability measure on $A$, denote by $\hat \mu:\hat A \to \C$ its Fourier transform $\hat \mu(\chi) = \int \chi(a) d\mu(a)$. By Fourier transform, the spectral gap constant $\kappa(\mu)$ is simply expressed in terms of $\hat \mu$ by
\begin{equation}\label{eq:spectral_gap_abelian}\kappa(\mu) = \max_{\chi \in \hat A \setminus \{1\}} \frac{1}{1-\Re\hat \mu(\chi)}.
\end{equation}
In the particular case of $A=(\Z/2\Z)^N$ and of uniform measures on finite sets, the spectral gap can be expressed in the language of error-correcting codes, and measures with small spectral gap constant (respectively small support) are the same as linear codes with large distance (respectively large dimension). We recall that in the next example, and refer for example to \cite{MR0465509} or \cite{MR1996953} for the vocabulary. I thank Jason Gaitonde for pointing out to me on Mathoverflow this connection with error-correcting codes \cite{414657}.
\begin{example}\label{ex:code_spectral_gap} Let $A=(\Z/2\Z)^N$ and let $\mu$ be the uniform probability measure on a generating family $a_1,\dots,a_K$. Define $C\subset (\Z/2\Z)^K$ the subgroup generated by $b_1,\dots,b_N$ where $b_j(i) = a_i(j)$. Then
  \begin{equation}\label{eq:relation_kappa_distance_binary} \kappa(\mu) = \frac{1}{2 \min\{ d_H(x,0) \mid x  \in C\setminus\{0\}\}}.
  \end{equation}
  Here $d_H$ is the normalized Hamming distance on $(\Z/2\Z)^K$
  \[ d_H(x,y) = \frac{1}{K}\sum_{i=1}^K 1_{x(i) \neq y(i)}.\]
  In the vocabulary of error correcting codes, $C$ is a $[K,N,\frac{K}{2\kappa(\mu)}]$-linear binary code. Here $K$ is the \emph{length}, $N$ is the \emph{dimension} and $\frac{K}{2\kappa(\mu)}$ is the \emph{distance} of the code (and $\frac 1{2\kappa(\mu)}$ is the \emph{relative distance}).

Conversely, if $C$ is a binary linear code of length $K$, dimension $N$ and distance $d$, any choice of a basis $b_1,\dots,b_N$ of $C$ gives rise to a subset of cardinality $K$ of $(\Z/2\Z)^N$ with spectral gap constant $2K/d$. More generally, let $q$ be a prime power and $C$ be a $[K,N,d]_q$-code, that is $C\subset \F_q^K$ is a linear subspace of dimension $N$ and distance $d = \min_{x \in C\setminus \{0\}} \#\{i\leq K\mid x_i \neq 0\}$. Any choice of a basis $b_1,\dots,b_N$ of $C$ allows us to define a subset of cardinality $(q-1)K$ of $\widehat{\F_q^N}$
  \[ \left\{ y \in \F_q^N \mapsto \chi(\sum_j y_j b_j(i) ) \mid \chi \in \widehat{\F_q}\setminus\{1\}, 1\leq i\leq K\right\}.\]
  The uniform  probability measure $\mu$ on this finite set has spectral gap constant
  \begin{equation}\label{eq:relation_kappa_distance_code}\kappa(\mu) = \frac{q-1}{q} \frac{K}{d}.
  \end{equation}
\end{example}
\begin{proof} In the construction of the binary code from the family $a_1,\dots,a_K$, the assumption that $a_1,\dots,a_K$ generates $(\Z/2\Z)^N$ is equivalent to $b_1,\dots,b_N$ being linearly free. So \eqref{eq:relation_kappa_distance_binary} is a particular case of \eqref{eq:relation_kappa_distance_code} for $q=2$, and all we have to do is justify \eqref{eq:relation_kappa_distance_code}. By \eqref{eq:spectral_gap_abelian}, we have
  \begin{align*} \frac{1}{\kappa(\mu)}& =\min_{y \in \F_q^N\setminus\{0\}}1-\frac{1}{(q-1)K} \sum_{\chi \in\hat{\F_q}\setminus\{0\}, i\leq K} \chi(\sum_j y_j b_j(i))\\
    & =\min_{y \in \F_q^N\setminus\{0\}}\frac{q}{q-1}(1-\frac{1}{K} \sum_{i\leq K} 1_{\sum_j y_j b_j(i)=0})\\
    & =\frac{q}{q-1}\frac{d}{K}.
  \end{align*}
  The second inequality is because, for every $z \in\F_q$,
  \[\frac{1}{q-1} \sum_{\chi \in\hat{\F_q}\setminus\{0\}} \chi(z) = \frac{1}{q-1}(-1+\sum_{\chi \in\hat{\F_q}}\chi(z)) = -\frac{1}{q-1} +\frac{q}{q-1} 1_{z=0}.\qedhere \]
\end{proof}

The following is a a result by Alon and Roichman \cite{MR1262979}, see also \cite{MR2097328} for a simple proof. By the previous example, it generalizes to arbitrary finite groups the classical fact that there exist \emph{asymptotically good binary linear codes}, that is linear codes with both dimension and distance proportional to the length. 



\begin{proposition}\cite{MR1262979}\label{prop:measure_with_small_support_and_small_Fourier_transform_nonabelian} There is a constant $C$ such that, for every finite group $G$, there is a subset $F \subset G$ of size $K\leq C  \log |G|$ such that $\mu$, the uniform probability measure on $F$, has spectral gap constant $\kappa(\mu)\leq 2$.
\end{proposition}
In this proposition, the logarithmic dependance between $G$ and $K$ is clearly optimal (if $G=(\Z/2\Z)^N$, $F$ has to contain a basis). But for interesting groups, (for example finite simple groups \cite{MR2221038}), much stronger results are known, and there exist measures of bounded support with uniform spectral gap.

\subsection{Average Gowers-Hatami theorem}
Some of the results in this subsection are probably well-known (except for Lemma~\ref{lem:GowersHatami_subgroup} that seems to be new), but we do not know whether they appear explicitely in the litterature. A form of the following theorem appears for example in \cite{vidickPauliBraiding}. The case $p=\infty$ is contained in \cite{MR3867328}, which was itself a generalization of \cite{MR3733361}. For unitary groups replaced by permutation groups, similar results also appear in \cite{beckerChapman}. The results of \cite{MR3867328} and \cite{beckerChapman} are also valid for amenable discrete groups. 

Following the terminology in \cite{MR3867328}, we call a unitarily invariant semi-norm on $\cM_\infty$ any semi-norm that is defined on an ideal of $\cM_\infty$ and that is invariant by left and right multiplication by unitaries, and that is by convention put to be equal to $\infty$ outside of this ideal. Later, we will only use the $2$-norm. Note that we do not ask that the ideal contains $\cM$.
\begin{theorem}\label{thm:averageGowersHatami} Let $G$ be a finite group, $1 \leq p < \infty$ and $\varphi:G \to \cU(\cM)$ be a map. Let $\|\cdot\|$ be a unitarily invariant semi-norm on $\cM_\infty$. Assume that
  \[ \left(\E_{g,h \in G} \|\varphi(gh)-\varphi(g)\varphi(h)\|^p\right)^{\frac 1 p} \leq \varepsilon.\]
  Then there is a projection $P \in \cM_\infty$, a unitary representation $\pi: G \to \cU(P\cM_\infty P)$ and an isometry $w \in P \cM_\infty 1_\cM$ such that
  \begin{equation}\label{eq:varphi_close_to_rep}\left(\E_g \|\varphi(g)-w^* \pi(g) w\|^p\right)^{\frac 1 p}\leq 13\varepsilon,
  \end{equation}
  \begin{equation}\label{eq:P_of_trace_almost_1} \|P-ww^*\| \leq 4\varepsilon.
  \end{equation}
\end{theorem}
By isometry we mean $w^*w=1_\cM$. As a consequence, if $\|\cdot\|$ is the norm in $L_q(\cM,\tau)$, \eqref{eq:P_of_trace_almost_1} is equivalent to $\tau(P) \leq 1+(4\varepsilon)^q$.
\begin{remark}\label{rem:added_by_referee}(suggested by the referee) It is possible to arrange $w=1_\cM$ in the conclusion of Theorem~\ref{thm:averageGowersHatami}, so that $P$ becomes a projection $\geq 1_\cM$ such that $\|P-1_\cM\|\leq 4\varepsilon$. Indeed, replace $P$ by $U^* P U$ and $\pi(g)$ by $U^* \pi(g) U$, where $U \in \cM_\infty$ is a unitary such that $w=PU 1_\cM$, for example $U=\begin{pmatrix} w & 1_\infty - ww^*\\1_\infty-w^*w& w^*\end{pmatrix} \in M_2(\cM_\infty) \simeq \cM_\infty$.
\end{remark}

This theorem follows by the same proof as in \cite{MR3867328}. We provide a detailed proof for completeness, with slightly better constants. We first prove an intermediary statement, with better constants but the isometry $w$ replaced by a contraction.

In the proof, we use two general facts about unitarily invariant norms \cite[Proposition 2.6]{MR3867328}:
\begin{equation}\label{eq:unitarilyinvariantnormC*} \forall X \in \cM, \|X^*X\|=\|XX^*\|,
\end{equation}
\begin{equation}\label{eq:unitarilyinvariantnormorder}
0 \leq A \leq B \implies \|A\|\leq \|B\|.
\end{equation}

\begin{lemma}\label{lemma:averageGowersHatami_non_isometry} Let $G$ be a finite group, $1 \leq p<\infty$ and $\varphi:G \to \cU(\cM)$ be a map. Let $\|\cdot\|$ be a unitarily invariant semi-norm on $\cM_\infty$. Assume that
  \[ \left(\E_{g,h \in G} \|\varphi(gh)-\varphi(g)\varphi(h)\|^p\right)^{\frac 1 p} \leq \varepsilon.\]
  Then there is a projection $P \in \cM_\infty$, a unitary representation $\pi : G \to \cU(P\cM_\infty P)$ and an element $X \in P \cM_\infty 1_\cM$ of operator norm $\|X\|_\infty\leq 1$ such that
  \begin{equation}\label{eq:varphi_close_to_repX} \left(\E_g \|\varphi(g)-X^* \pi(g) X\|^p\right)^{\frac 1 p} \leq 5 \varepsilon,
  \end{equation}
\begin{equation}\label{eq:XstarX_close_to1} \|1_\cM-X^*X\| \leq 4 \varepsilon,
\end{equation}
and
\begin{equation}\label{eq:XXstar_close_toP} \|P-XX^*\| \leq 4 \varepsilon,
\end{equation}
\end{lemma}
\begin{proof} Let $\cH$ be the Hilbert space on which $\cM$ is realized. Define $V \colon \cH \to L_2(G,\cH)$ by $(V\xi)(g)  = \varphi(g^{-1})\xi$. It is clear that $V x= (1_{L_2(G)}\otimes x) V$ for every $x \in \cM'$, so (if we identify $L_2(G)$ with a subspace of $\ell_2$ with orthogonal projection $Q$), $V$ is an isometry in $Q \cM_\infty 1_\cM$. Its adjoint is $V^*f = \E_g \varphi(g^{-1})^* f(g)$, so that
   \[ V^* (\lambda(g)\otimes 1_\cH) V = \E_h \varphi(h)^* \varphi(h g)=: \tilde \varphi(g).\]
   Observe that
\[ \E_g \|\varphi(g) - \tilde\varphi(g)\|^p \leq \E_{g,h} \|\varphi(g) - \varphi(h)^* \varphi(hg)\|^p \leq \varepsilon^p.\]
We are almost done, except that $\tau_\infty(Q) = |G|$. If $VV^*$ did commute with the representation $\lambda\otimes 1_\cH$, we would be done with $X=V$ and $P=VV^*$ and $\pi(g)= (\lambda(g)\otimes 1_\cH) VV^*$. In general, we define $X=PV$ and $\pi$ to be the restriction of $\lambda \otimes 1_{\cH}$ on the image of $P$, where $P$ is the spectral projection $\chi_{[1/2,1]}(A)$, for $A$ the conditional expectation of $VV^*$ on $\cM \otimes \lambda(G)'$, that is
\[ A:= \E_g \lambda(g) VV^* \lambda(g^{-1}).\]
It remains to justify \eqref{eq:varphi_close_to_repX}, \eqref{eq:XstarX_close_to1} and \eqref{eq:XXstar_close_toP}. To do so we first observe that, writing $1-\tilde{\varphi}(g)^* \tilde\varphi(g) = \varphi(g)^*(\varphi(g) - \tilde{\varphi}(g)) + (\varphi(g) - \tilde{\varphi}(g))^* \tilde{\varphi}(g)$, we can bound by H\"older's inequality
\begin{equation}\label{eq:tildevarphi_close_to_1} \E_g \|1-\tilde{\varphi}(g)^* \tilde\varphi(g)\| \leq 2\E_g \|\varphi(g) - \tilde{\varphi}(g)\| \leq 2{\varepsilon}.
\end{equation}
We deduce
   \[ \| 1-V^*A V\| = \|\E_g (1-\tilde\varphi (g)^*\tilde\varphi(g))\| \leq 2{\varepsilon}.\]
   Therefore (using that $\lambda(g)$ commutes with $\sqrt{1-A}$ and \eqref{eq:unitarilyinvariantnormC*})
   \begin{align*} \|A - A^2\| &= \| \sqrt{1-A} \E_g \lambda(g) VV^* \lambda(g) \sqrt{1-A}\|\\
   & \leq \| \sqrt{1-A} V V^* \sqrt{1-A}\| = \|V^*(1-A) V\| \leq 2{\varepsilon}.
   \end{align*}
We now prove \eqref{eq:XstarX_close_to1}. Using $(1-P) \leq 2(1-A)$ and \eqref{eq:unitarilyinvariantnormorder}, we have
   \[ \|1_\cM - X^*X\| = \|V^*(1-P)V\| \leq 2 \|V^* (1-A) V\| \leq 4 {\varepsilon}.\]

   We can now turn to \eqref{eq:varphi_close_to_repX}. By the triangle inequality,
   \begin{multline*} \left(\E_g \|\varphi(g)-X^* \pi(g) X\|^p\right)^{\frac 1 p} \\\leq \left(\E_g \|\varphi(g)-\widetilde{\varphi}(g)\|^p\right)^{\frac 1 p} + \left(\E_g \|\widetilde \varphi(g)-X^* \pi(g) X\|^p\right)^{\frac 1 p}.
   \end{multline*}
   The first term is $\leq \varepsilon$. The second term is (by \cite[Corollary 2.8]{MR3867328})
   \begin{align*} \left(\E_g \|V^*(1-P) \lambda(g)V\|^p\right)^{\frac 1 p} & \leq \|V^*(1-P)V\| \leq 4\varepsilon.
   \end{align*}
The last inequality is the already proven \eqref{eq:XstarX_close_to1}.

Finally, we prove \eqref{eq:XXstar_close_toP}. Using that $P \leq 2A$,
   \begin{align*} \|P - XX^*\| &= \|P(1-VV^*)P\| \\&= \|(1-VV^*)P(1-VV^*)\|\\ &\leq 2 \|(1-VV^*) A (1-VV^*)\|.
   \end{align*}
   By the triangle inequality, this is less than
   \begin{multline*} 2 \E_g \|(1-VV^*) \lambda(g) VV^* \lambda(g)^* (1-VV^*)\|\\ = 2 \E_g \|V^* \lambda(g)^* (1-VV^*) \lambda(g) V\| = 2 \E_g \|1-\tilde{\varphi}(g)^* \tilde\varphi(g)\|
   \end{multline*}
   which is less than $4\varepsilon$ by \eqref{eq:tildevarphi_close_to_1}.
\end{proof}
\begin{proof}[Proof of Theorem~\ref{thm:averageGowersHatami}]
Let $X,P,\pi$ be given by Lemma~\ref{lemma:averageGowersHatami_non_isometry}. Write $X = w_0 |X|$ the polar decomposition of $X$. We can find a partial isometry such $w_1 \in \cM_\infty$ such that $w_0^* w_0 +w_1 ^* w_1=1_\cM$ and $w_1 w_1^*$ is orthogonal to $P$. Define $P'=P+w_1 w_1^*$, and extend $\pi$ to a unitary representation on $P'\cM_\infty P'$ by declaring that $\pi(g)$ is the identity on $w_1 w_1^*$. If we define $w=w_0+w_1$, $w$ is then an isometry in $P' \cM_\infty 1_\cM$ such that $X=w |X|$.
  We can then bound
  \[ \| w-X\| = \|1_\cM - |X|\| \leq \|1_\cM - X^*X\| \leq 4\varepsilon,\]
  where the first inequality \eqref{eq:unitarilyinvariantnormorder} is because $1-t \leq 1-t^2$ for every $0 \leq t \leq 1$, and the second inequality is \eqref{eq:XstarX_close_to1}. Moreover, by \eqref{eq:XXstar_close_toP} we have
  \[ \|P'-ww^*\| = \|P-w_0 w_0^*\| \leq \|P-X X^*\| \leq 4\varepsilon.\]

  Therefore, we obtain
  \[ \left(\E_g \|\varphi(g)-w^* \pi(g) w\|^p\right)^{\frac 1 p} \leq \left(\E_g \|\varphi(g)-X^* \pi(g) X\|^p\right)^{\frac 1 p} + 2 \|X-w\|.\]
  This is less than $13 \sqrt{\varepsilon}$ by \eqref{eq:varphi_close_to_repX}.
\end{proof}
The next lemma shows that, if $\varphi$ is assumed to behave well on a subgroup, then $\varphi$ and $w^* \pi w$ are close on this subgroup. The lemma is clearly false without the additional assumption \eqref{eq:phi_H_left_equivariant} or \eqref{eq:phi_H_right_equivariant} : if $H$ has very small index, perturbing arbitrarily $\varphi$ on $H$ will not affect much the hypothesis nor the conclusion of Theorem~\ref{thm:averageGowersHatami}, but the conclusion of the lemma cannot hold.
\begin{lemma}\label{lem:GowersHatami_subgroup} Let $G$, $\varphi$, $\varepsilon$, $P$, $\pi$ and $w$ be as in the hypothesis and conclusion of Theorem~\ref{thm:averageGowersHatami}. Let $H<G$ is a subgroup, and assume either
  \begin{equation}\label{eq:phi_H_left_equivariant} \forall h\in H, \forall g\in G,\  \varphi(hg) = \varphi(h) \varphi(g)
  \end{equation}
  or
  \begin{equation}\label{eq:phi_H_right_equivariant}\forall h\in H, \forall g\in G,\  \varphi(gh) = \varphi(g) \varphi(h).
  \end{equation}
  Then
  \[ \left(\E_{h \in H} \|\varphi(h)-w^* \pi(h) w\|^p\right)^{\frac 1 p}< 38 \varepsilon.\]
\end{lemma}
An inspection of the proof reveals that the statement holds (with $38$ replaced by another constant) if \eqref{eq:phi_H_left_equivariant} is replaced by the weaker hypothesis that
\[\left(\E_{h\in H,g \in G} \|\varphi(hg) - \varphi(h) \varphi(g)\|^p\right)^{\frac 1 p} + \left(\E_{h \in H, h'\in H} \|\varphi(hh') - \varphi(h) \varphi(h')\|^p\right)^{\frac 1 p} = O(\varepsilon). \]
\begin{proof}
Assume \eqref{eq:phi_H_left_equivariant}. The case when we assume \eqref{eq:phi_H_right_equivariant} is proved in the same way and left to the reader. Define $\psi(g) = w^* \pi(g) w$. This is not a representation of $G$, but almost: for every $g_1,g_2 \in G$,
  \begin{multline}\label{eq:psi_almost_rep} \| \psi(g_1g_2) - \psi(g_1)\psi(g_2)\| \\= \| w^* \pi(g_1) (P-ww^*) \pi(g_2) w\| \leq \|P-ww^*\| \leq 4\varepsilon.
  \end{multline}
the last inequality is the assumption~\eqref{eq:P_of_trace_almost_1}. So it follows from \eqref{eq:phi_H_left_equivariant} and the triangle inequality that
  \begin{multline*} (\E_{g \in G} \E_{h \in H} \|\varphi(h)\varphi(g) - \psi(h) \psi(g)\|^p)^{\frac 1 p} \\\leq 4\varepsilon + (\E_{g \in G} \E_{h \in H} \|\varphi(hg) - \psi(hg)\|^p)^{\frac 1 p} \leq 17 {\varepsilon}.
  \end{multline*}
The last inequality is \eqref{eq:varphi_close_to_rep} because, for fixed $h \in H$, $hg$ is uniformly distributed in $G$ when $g$ is. In particular, there is $X$ of operator norm $\leq 1$ (namely $X=\psi(g) \varphi(g)^{-1}$ for some $g$) such that
  \begin{equation}\label{eq:diff_phi_psiX} (\E_{h \in H} \|\varphi(h) - \psi(h) X\|^p)^{\frac 1 p} \leq 17 {\varepsilon}.
  \end{equation}
  We deduce
  \begin{align*}
    (\E_h \|\varphi(h) - \psi(h)\|^p)^{\frac 1 p} & = (\E_{h_1,h_2} \|\varphi(h_1)\varphi(h_2) - \psi(h_1) \varphi(h_2)\|^p)^{\frac 1 p} \\
    & \leq 17{\varepsilon}+ (\E_{h_1,h_2} \|\varphi(h_1h_2) - \psi(h_1) \psi(h_2) X\|^p)^{\frac 1 p}\\
    &\leq 21{\varepsilon}+ (\E_{h_1,h_2} \|\varphi(h_1h_2) - \psi(h_1h_2)X\|^p)^{\frac 1 2}\\
    &\leq 38{\varepsilon}.
  \end{align*}
    The first inequality is the unitary invariance of the $\|\cdot\|$-norm, the second is \eqref{eq:phi_H_left_equivariant} and \eqref{eq:diff_phi_psiX}, the third is \eqref{eq:psi_almost_rep} and the last is \eqref{eq:diff_phi_psiX} again. This concludes the proof of the lemma.
\end{proof}
We will use the following consequences. For simplicity we restrict to the $2$-norm.
\begin{corollary}\label{cor:averageGowersHatami_products} Let $A,B$ be two finite groups and $U:A \to \cU(\cM)$ and $V : B \to \cU(\cM)$ be two group homomorphisms. If they satisfy
  \[ \E_{a \in A,b \in B} \|[U(a),V(b)]\|_2^2 \leq \varepsilon,\]
  then there is a projection $P \in \cM_\infty$, unitary representation $\tilde{U}: A \to \cU(P\cM_\infty P)$ and $\tilde{V}: B \to \cU(P\cM_\infty P)$ with commuting ranges, and an isometry $w \in P \cM_\infty 1_\cM$ such that
  \[\E_{a \in A} \|U(a)-w^* \tilde U(a) w\|_2^2<1444\varepsilon,\]
  \[\E_{b \in B} \|U(b)-w^* \tilde V(b) w\|_2^2<1444\varepsilon,\]
  \[ \tau_\infty(P)\leq 1+16\varepsilon.\]
\end{corollary}
\begin{proof} Consider $G=A \times B$ and define $\varphi:A \times B \to \cU(\cM)$ by $\varphi(a,b) = U(a) V(b)$. Then for $g=(a,b)$ and $h=(a',b')$, by unitary invariance of the $\|\cdot\|_2$-norm
  \begin{align*} \|\varphi(gh) - \varphi(g) \varphi(h)\|_2 &= \| U(aa') V(bb') - U(a) V(b) U(a') V(b')\|_2\\ &= \|[U(a'),V(b)]\|_2,
  \end{align*}
  so the assumption is exactly that 
  \[ \E_{g,h} \|\varphi(gh) - \varphi(g) \varphi(h)\|_2^2 \leq \varepsilon.\]
  By applying Theorem~\ref{thm:averageGowersHatami} for the norm on $L_2(\cM_\infty,\tau_\infty)$ and $p=2$, we obtain a projection $P$ and, an isometry $w$ and a representation $\pi:A \times B \to \cU(P\cM_\infty P)$ satisfying
  \[ \tau(P) \leq 1+16 \varepsilon\]
  and
  \begin{equation}\label{eq:phi_close_to_rho} \E_g \|\varphi(g) - w^* \pi(g) w\|_2^2 \leq 169 \varepsilon.\end{equation}
  If we define $\tilde{U}(a) = \pi(a,1)$ and $\tilde{V}(b) = \pi(1,b)$, these are unitary representations with commuting ranges. Moreover, since $\varphi(a g) =\varphi(a) \varphi(g)$ for every $a \in A$, $g \in G$, and similarly for $B$ on the right, Lemma~\ref{lem:GowersHatami_subgroup} concludes the proof.
\end{proof}

\begin{corollary}\label{cor:averageGowersHatami_central_ext_products} Let $A,B$ be two finite groups and $\gamma:A\times B \to \{-1,1\}$ a map such that $\gamma(a_1 a_2,b_1b_2) = \prod_{i,j} \gamma(a_i, b_j)$.\footnote{This is the same as $\gamma$ being a $2$-cocycle on $A \times B$ whose restriction to $A$ and $B$ is zero, or a group homomorphism $A \to \mathrm{hom}(B,\{-1,1\})$, or a group homomorphism $B \to \mathrm{hom}(A,\{-1,1\})$.} Let $U:A \to \cU(\cM)$ and $V : B \to \cU(\cM)$ be two group homomorphisms. Assume that they satisfy
  \[ \E_{a \in A,b \in B} \|U(a)V(b) - \gamma(a,b) V(b) U(a)\|_2^2 \leq \varepsilon.\]
There is a projection $P \in \cM_\infty$, unitary representations $\tilde U:A \to \cU(P\cM_\infty P)$ and $\tilde V:B \to \cU(P\cM_\infty P)$ and a partial isometry $w \in P \cM_\infty 1_\cM$ such that
  \[ \tilde U(a) \tilde V(b) = \gamma(a,b) \tilde V(b) \tilde U(a) \forall a\in A, b \in B,\]
  \[\E_{a \in A} \|U(a)-w^* \tilde U(a) w\|_2^2< 30000 \varepsilon,\]
  \[\E_{b \in B} \|U(b)-w^* \tilde V(b) w\|_2^2< 30000 \varepsilon,\]
  \[ \tau_\infty(P)\leq 1+16\varepsilon.\]
\end{corollary}
\begin{proof} Consider $G$, the central extension of $A\times B$ by $\{-1,1\}$ given by $\gamma$. This is a Weyl-Heisenberg-type group, that is the set $A\times B \times \Z/2\Z$ for the group operation
  \[ (a,b,z)(a',b',z') = (aa',bb',\gamma(b,a')zz').\]
  We shall identify $A$, $B$ and $\{-1,1\}$ with the subgroups $\{(a,1,1)\mid a \in A\}$,  $\{(1,b,1)\mid b \in b\}$ and  $\{(1,1,1),(1,1,-1)\}$. 
  Define $\varphi:G \to \cU(\cM)$ by
  \[\varphi(a,b,z) = z U(a) V(b).\] Clearly, it satisfies
  \[ \varphi(a g) = \varphi(a) \varphi(g), \varphi(gb) = \varphi(g) \varphi(b), \varphi(zg) = \varphi(z) \varphi(g)\]
  for every $g \in G$, $a \in A$, $b \in B$, $z \in \{-1,1\}$. Moreover our assumption is equivalent to
  \[ \E_g \|\varphi(gh) - \varphi(g) \varphi(h)\|_2^2 \leq \varepsilon.\]
  We can apply Theorem~\ref{thm:averageGowersHatami}for the $\|\cdot\|_2$-norm and $p=2$ to obtain a projection $P$, a representation $\pi:G \to \cU(P\cM_\infty P)$ and an isometry $w \in P \cM_\infty 1_\cM$ as in Theorem~\ref{thm:averageGowersHatami}. Also, by Lemma~\ref{lem:GowersHatami_subgroup} we know that, for $H=A,B$ or $\{-1,1\}$.
  \begin{equation}\label{eq:varphi_close_to_rho_on_subgroups} \E_{h\in H} \|\varphi(h) - w^* \pi(h) w\|_2^2 \leq 1444 \varepsilon.
  \end{equation}
We are almost at the desired conclusion, except that we do not a priori have the desired (anti)-commutation relation, but only (denoting $Z = \pi(1,1,-1)$)
\[ \pi(a) \pi (b) = Z \pi(b) \pi(a)\textrm{ when }\gamma(a,b)=-1.\]  
So we would be done if $Z$ was equal to $-P$. This is almost the case. Indeed, taking $H = \{-1,1\}$ in \eqref{eq:varphi_close_to_rho_on_subgroups} we see
\[ \|1_\cM + w^* Z w\|_2^2 \leq 2888\varepsilon\]
and we deduce
\[ \| P + Z\|_2 \leq \|ww^* (P+Z)ww^*\|_2 + 2\|P-ww^*\|_2 \leq 38\sqrt{2\varepsilon}+8\sqrt{\varepsilon}.\]
But $Z$ is an order $2$ unitary in $P \cM_\infty P$, so it can be written as $P-2Q$ for a projection $Q \leq P$, which commutes with the range of $\pi$ because $(1,1,-1)$ belongs to the center of $G$. And we have just proved that $\|P-Q\|_2 \leq (19\sqrt{2}+4)\sqrt \varepsilon$. So let us replace $P$ by $Q$, $\pi$ by $g \mapsto Q \pi(g)$ and $w$ by $w'$, the partial isometry in the polar decomposition of $Qw$. We claim that
\[ \|w-w'\|_2 \leq \|w - Qw\|_2 + \|w'-Qw\|_2\leq 2 \|P-Q\|_2.\]
Indeed, the first term is $\|Pw-Qw\|_2 \leq \|P-Q\|_2$, and the second term is
\[ \|(w')^*w' - |Qw|\|_2 \leq \|1_\cM - |Qw|\|_2 \leq \|1_\cM -|Qw|^2\| = \|w^*(P-Q)w\|_2 \leq \|P-Q\|_2.\]
So we deduce for $H=A$ or $H=B$
\begin{align*} \left(\E_{h \in H} \|\varphi(h) - (w')^* \pi(h) w'\|_2^2\right)^{\frac 1 2} &\leq \left(\E_a \|\varphi(h) - w^* \pi(h) w\|_2^2\right)^{\frac 1 2} + 2\|w-w'\|_2\\
  & \leq (38+4(19\sqrt{2}+4))\sqrt{\varepsilon}.
\end{align*}
The last inequality is \eqref{eq:varphi_close_to_rho_on_subgroups}. This proves the corollary.  
\end{proof}

\section{The main result}\label{sec:main}
The main new contribution in this note is the following result. It is expressed in term of the spectral gaps $\kappa(\cdot)$ defined in subsection~\ref{sec:spectral_gap}. 
\begin{theorem}\label{thm:almost_commutation}
  Let $A,B$ be two finite groups equipped with probability measures $\mu,\nu$.
  
Let $U:A \to \cU(\cM)$ and $V:B \to \cU(\cM)$ be two homomorphisms. Then
  \[ \E_{a \in A,b\in B} \| [U(a),V(b)]\|_2^2 \leq \kappa(\mu)\kappa(\nu) \int \| [U(a),V(b)]\|_2^2 d\mu(a) d\nu(b).\]
\end{theorem}

We first prove a result with only one group entering the picture.
\begin{lemma}\label{lemma:spectral_gap_commutator} Let $G$ be a group, $\mu$ a probability measure on $G$ with generating support, and $U\colon G\to \cU(\cM)$ be a group homomorphism. Define $\cN=\cM \cap U(G)'$ and $E_\cN:\cM\to\cN$ the conditional expectation.

  For every $V \in \cM$,
  \begin{equation}\label{eq:spectral_gap_and_commutators} \| V - E_\cN(V)\|_2^2 \leq \frac{\kappa(\mu)}{2} \int \|[U(g),V]\|_2^2 d\mu(g).
  \end{equation}

  If $G$ is finite,
  \[ \E_{g \in G} \|[U(g),V]\|_2^2 \leq \kappa(\mu) \int \|[U(g),V]\|_2^2 d\mu(g).\]
\end{lemma}
\begin{proof}   The formula
  \[ \pi(a) X = U(a) X U(a)^*\]
  defines a unitary representation of $\pi$ on the Hilbert space $L_2(\cM,\tau)$. Here we use in an essential way that $\tau$ is trace. Moreover, the space of invariant vectors $L_2(\cM,\tau)^\pi$ coincides with $L_2(\cN,\tau)$, and the orthogonal projection on $L_2(\cM,\tau)^\pi$ is just $E_\cN$. Therefore, \eqref{eq:spectral_gap_and_commutators} is exactly the Poincar\'e inequality \eqref{eq:Poincare_constant}.

  If $G$ is finite, $E_\cN$ is given by $V\mapsto \E_{g\in G} U(g) V U(g)^*$, so
  \begin{align*} \E_{g \in G} \|[U(g),V]\|_2^2 &= 2\|V\|_2^2-2\tau(\E_{g \in G} U(g)^* V U(g)V^*)\\& = 2\|V\|_2^2-2\|E_\cN(V)\|_2^2 = 2\|V-E_\cN(V)\|_2^2.
  \end{align*}
  So the second inequality is a particular case of \eqref{eq:spectral_gap_and_commutators}.
\end{proof}

\begin{proof}[Proof of Theorem~\ref{sec:main}] Successive applications of Lemma~\ref{lemma:spectral_gap_commutator} for $A$ and for $B$ give
  \begin{align*}\E_{a,b} \|[U(a),V(b)]\|_2^2 & \leq \E_b \kappa(\mu) \int_A \|[U(a),V(b)]\|_2^2 d\mu(a)\\
    & = \kappa(\mu)   \int_A \E_b \|[U(a),V(b)]\|_2^2 d\mu(a)\\
    & \leq \kappa(\mu)   \int_A \kappa(\nu) \int_B  \|[U(a),V(b)]\|_2^2 d\nu(b) d\mu(a).
  \end{align*}
  This proves the theorem.
\end{proof}
\subsection{Consequences}
We list four consequences of this theorem. Only the last one (Corollary~\ref{thm:almost_anticommutation_2group}) will be used in the next section.
\begin{corollary} Let $A,B,U,V,\mu,\nu$ be as in Theorem~\ref{thm:almost_commutation}. If
  \[ \int \| [U(a),V(b)]\|_2^2 d\mu(a) d\nu(b) \leq \varepsilon,\]
 then there is a projection $P \in \cM_\infty$, unitary representations $\tilde{U}: A \to \cU(P\cM_\infty P)$ and $\tilde{V}: B \to \cU(P\cM_\infty P)$ with commuting ranges, and an isometry $w \in P \cM_\infty 1_\cM$ such that
  \[\E_{a \in A} \|U(a)-w^* \tilde U(a) w\|_2^2\lesssim  \kappa(\mu)\kappa(\nu)\varepsilon,\]
  \[\E_{b \in B} \|U(b)-w^* \tilde V(b) w\|_2^2\lesssim  \kappa(\mu)\kappa(\nu)\varepsilon,\]
  \[ \tau_\infty(P)\leq 1+  16\kappa(\mu)\kappa(\nu) \varepsilon.\]
\end{corollary}
\begin{proof}
This follows from Theorem~\ref{thm:almost_commutation} and
Corollary~\ref{cor:averageGowersHatami_products}.\end{proof} For
abelian groups, we have a stronger conclusion.
\begin{corollary}\label{thm:almost_commutationbis}
  Let $A,B,U,V,\mu,\nu$ be as in Theorem~\ref{thm:almost_commutation}, with $A$ and $B$ abelian. If
  \[ \int \|[U(a),V(b)]\|_2^2  d\mu(a) d\nu(b) \leq \varepsilon,\]
  then there exists a unitary representation $V_1:B \to \cU(\cM)$ such that
  \[ [U(a),V_1(b)] = 0 \ \ \forall a \in A,b\in B\]
  and
  \[ \E_b \|V(b) - V_1(b)\|^2 \leq 10\kappa(\mu) \kappa(\nu) \varepsilon.\]
\end{corollary}
\begin{proof} By Theorem~\ref{thm:almost_commutation}, the hypothesis implies that
  \begin{equation}\label{eq:U_V_almost_commute} \E_{a,b} \| [U(a),V(b)]\|_2^2 \leq \kappa(\mu) \kappa(\nu) \varepsilon.
  \end{equation}
Therefore, the result follows from \cite[Corollary 5]{orthonormalisation}. Observe that \cite[Corollary 5]{orthonormalisation} is only stated for $A=\Z/n\Z$ and $B=\Z/m\Z$, but its proof is valid for any finite abelian groups. Indeed, by Lemma~\ref{lem:Fourier_abelian_group}, \eqref{eq:U_V_almost_commute} can be translated to the PVMs corresponding to $U$ and $V$ being close, so \cite[Theorem 4]{orthonormalisation} applies.
\end{proof}
\begin{corollary}\label{thm:almost_anticommutation}
Let $A$ be a finite abelian group, and $\mu,\nu$ be probability measures on $A$ and $\hat A$ respectively.

Let $U:A \to \cU(\cM)$ and $V:\hat A \to \cU(\cM)$ be two homomorphisms satisfying
  \[ \int \| U(a)V(\chi) -\chi(a) V(\chi) U(a) \|_2^2 d\mu(a) d\nu(\chi) \leq \varepsilon.\]
Then they satisfy
    \[ \E_{a,\chi}  \| U(a)V(\chi) -\chi(a) V(\chi) U(a) \|_2^2 \leq \kappa(\mu)\kappa(\nu)\varepsilon.\]
\end{corollary}
\begin{proof} 
  Consider, on $\ell_2(A)$, the operators (called Pauli matrices) $\lambda(a)$ of translation operators, and $M(\chi)$ of multiplication by $\chi$:
  \[(\lambda(a) f)(a') = f(a^{-1}a'), (M(\chi) f)(a') = \chi(a') f(a').\]
  They satisfy $\lambda(a) M(\chi) = \overline{\chi}(a) M(\chi) \lambda(a)$.
  Therefore, on $\cM \otimes M_A(\C)$ with the trace $\tau \otimes \tr$, if we define
  \[ \tilde{U}(a) = U(a) \otimes \lambda(a), \tilde{V}(\chi) = V(\chi) \otimes \cM(\chi),\]
  we have
  \[ \tilde U(a)\tilde V(\chi) -\tilde  V(\chi) \tilde U(a) = (U(a)V(\chi) -\chi(a) V(\chi) U(a))\otimes \lambda(a) M(\chi).\]
  So we conclude by Theorem~\ref{thm:almost_commutation}.
\end{proof}
\begin{corollary}\label{thm:almost_anticommutation_2group} Assume, in addition to the hypotheses of Corollary~\ref{thm:almost_anticommutation}, that $A$ is a $2$-group. Then there is a projection $P \in \cM_\infty$, a partial isometry $w \in P \cM_\infty 1_\cM$, a pair of representations $U_1:A \to \cU(P\cM_\infty P)$ and $V_1:\hat A \to \cU(P\cM_\infty P)$ that satisfy
  \[ U_1(a) V_1(\chi) = \chi(a) V_1(\chi) U_1(a)\forall a \in A,\chi \in \hat A,\]
  \[ \E_a \| U(a) - w^* U_1(a)w\|_2^2 \lesssim \kappa(\mu)\kappa(\nu) \varepsilon,\]
  \[ \E_b \| V(b) - w^* V_1(b)w\|_2^2 \lesssim \kappa(\mu)\kappa(\nu) \varepsilon,\]
  \[ \|1-w^*w\|_2^2 \lesssim \kappa(\mu)\kappa(\nu) \varepsilon,\]
  \[ \|P-ww^*\|_2^2 \lesssim \kappa(\mu)\kappa(\nu) \varepsilon.\]
\end{corollary}
\begin{proof} Combine Corollary~\ref{thm:almost_anticommutation} with Corollary~\ref{cor:averageGowersHatami_central_ext_products}.
\end{proof}

\section{Translation to games and strategies}\label{section:games}
In this section, we translate the results of the previous section (in particular Corollary~\ref{thm:almost_anticommutation_2group}) in terms of (two-player one-round) games. This translation is adapted from a construction by Natarajan and Vidick \cite{MR3678246}, which was also adapted in similar fashion in \cite{MIPRE} (see \S~\ref{subsection:comparison} for a comparison).

Since we only consider synchronous strategies, we adopt the following slightly unconventional definition of a game, but that is equivalent to the standard notion of a synchronous game for our purposes.
\begin{definition}\label{def:game} A game is a tuple $G=(X,\mu,A,D)$ where $X$ is a finite set, $\mu$ is a probability distribution on $X\times X$, $A=(A(x))_{x \in X}$ is a collection of finite sets (the possible answers to question $x$) and $D$ is a $\{0,1\}$-valued function defined on $\{(x,y,a,b)\mid (x,y) \in \mathrm{supp}(\mu), a \in A(x), b \in A(y)\}$ and that is required to be symmetric (that is $D(x,y,a,b)=D(y,x,b,a)$ whenever both terms are defined).  
\end{definition}
By abuse, we will also write by the same symbol $\mu$ the probability measure on $X$ given by
\[\mu(x) = \frac{1}{2} \sum_{y \in X} \mu(x,y)+\mu(y,x).\]

If $G=(X,\mu,A,D)$ is a game, a synchronous strategy in $(\cM,\tau)$ is, for every $x \in X$, a PVM $(P^x_a)_{a \in A(x)}$ in $\cM$. The value of the game on this strategy is $\val_G(P):=\int \sum_{a \in A(x), b \in A(y)} D(x,y,a,b) \tau(P^x_a P^y_b) d\mu(x,y)$.

\begin{remark} In the standard definition of a synchronous game, $\mu$ and $D$ are both required to be symmetric, and it is moreover required that $D(x,x,a,b)=0$ whenever $a \neq b$. This last requirement is not important for us because the value of a synchronous strategy does not involve such values of $D$ (because $\tau(P^x_a P^x_b)=0$ if $a\neq b$). Moreover, a game as defined in Definition~\ref{def:game} can be turned a symmetric game as in the standard definition by replacing $\mu$ by $\frac{1}{2}(\mu + \tilde{\mu})$, where $\tilde{\mu}(x,y)=\mu(y,x)$ and extending $D$ by symmetry. But this is not very important, as the value of any synchronous strategy for the original game and the symmetrized game agree.
\end{remark}
\begin{definition} We say that a synchronous strategy $(Q^x_a)_{a \in A(x)}^{x \in X}$ in $(\cN,\tau')$ is $\varepsilon$-close to another synchronous strategy $(P^x_a)_{a \in A(x)}^{x \in X}$ in $(\cM,\tau)$ if~:
  \begin{itemize}
    \item There is a projection $P \in \cM_\infty$ of finite trace such that $\cN = P \cM_\infty P$ with trace $\tau' = \frac{1}{\tau_\infty(P)} \tau$.
    \item There is a partial isometry $w \in P \cM_\infty 1_\cM$ such that
      \[ \tau(1-w^*w) \leq \varepsilon, \tau'(P-w w^*)\leq \varepsilon,\]
      \item $\E_x \sum_{a \in A(x)} \|P^x_a - w^* Q^x_a w\|_2^2\leq \varepsilon$.
  \end{itemize}
\end{definition}
The following straightforward lemma illustrates that this notion of strategies being close is compatible with the closeness of unitary representations that was considered in the previous sections.
\begin{lemma}\label{lem:close_strategies_unitaries} Let $(\cM,\tau)$ be a tracial von Neumann algebra, $P \in \cM_\infty$ be a projection and $w \in P\cM_\infty 1_\cM$. Let $H$ be a finite abelian group, $U:H\to \cU(\cM)$ and $V:H \to\cU(P\cM_\infty P)$ be two unitary representations, with corresponding PVMs $(P_\chi)_{\chi \in  \hat H}$ and $(Q_\chi)_{\chi \in  \hat H}$ as in Lemma~\ref{lem:Fourier_abelian_group}. Then
  \[ \E_{h\in H} \|U(h) - w^* V(h) w\|_2^2 = \sum_{\chi \in \hat H} \|P_\chi - w^* Q_\chi w\|_2^2.\]
\end{lemma}
\begin{proof}
  The left-hand side is
  \[ \E_{h\in H} \|\sum_\chi \chi(h) (P_\chi -  w^* Q_\chi w)\|_2^2,\]
  so the lemma is just the orthogonality of characters.\end{proof}

We start with two important and well-known examples. I include proofs for completeness, but there is nothing original here.
\subsection{The commutation game}
Given two sets $A_1,A_2$, the \emph{commutation game} on $A_1,A_2$ is the game where
\begin{itemize}
\item $X=\{x_1,x_2,y\}$, $A(x_i) = A_i$ and $A(y) = A_1 \times A_2$,
\item $\mu = \frac 1 2 ( \delta_{(x_1,y)}+  \delta_{(x_2,y)})$,
\item $D(x_1,y,a,(a',b')) = 1_{a=a'}$ and $D(x_2,y,b,(a',b')) = 1_{b=b'}$.
\end{itemize}
The important feature of this game is that strategies with large value have to be almost commuting, as shown by the following lemma. For later use, we denote $G_{com}= (X_{com},A_{com},\mu_{com},D_{com})$ the commutation game on $A_1=A_2=\{-1,1\}$, and $x_{com,1}=x_1$ and $x_{com,2}=x_2$.
\begin{lemma}\label{lem:commutation_game} If a synchronous strategy achieves the value $1-\varepsilon$ on the commutation game of $A_1,A_2$, then its restriction $(p_a)_{a \in A_1}$ and $(q_b)_{b \in A_2}$ to $x_1$ and $x_2$ respectively satisfies
  \begin{equation}\label{eq:commutation_projection}\sum_{a\in A_1,b \in A_2} \|[p_a,q_b]^2\|_2^2 \leq 16 \varepsilon.
  \end{equation}
  If $A_1=A_2=\{-1,1\}$ then
  \begin{equation}\label{eq:commutation_unitary} \|[p_1-p_{-1},q_1-q_{-1}]\|_2^2 \leq 64 \varepsilon.
  \end{equation}
\end{lemma}
\begin{proof}
  The assumption means that there is a PVM $(r_{a,b})_{(a,b)\in A_1 \times A_2}$ such that
  \[ \frac 1 2 \sum_{a,b} \tau( p_a r_{a,b}) + \tau(q_b r_{a,b}) \geq 1-\varepsilon.\]
Define $p'_a= \sum_b r_{a,b}$ and $q'_b = \sum_a r_{a,b}$. The previous inequality can be equivalently written as $\eta_1+\eta_2 \leq 4\varepsilon$, where
  \[ \eta_1 = \sum_a \|p_a - p'_a\|_2^2\]
  and
  \[ \eta_2 = \sum_b \|q_b - q'_b\|_2^2.\]
 We can bound $(\sum_{a,b} \|[p_a,q_b]^2\|_2^2)^{\frac 1 2}$ by
  \[ (\sum_{a,b} \|[p_a-p'_a,q_b]^2\|_2^2)^{\frac 1 2} + (\sum_{a,b} \|[p'_a,q_b-q'_b]^2\|_2^2)^{\frac 1 2} + (\sum_{a,b} \|[p'_a,q'_b]^2\|_2^2)^{\frac 1 2}.\]
Using the easy bound
\[ \sum_b \|[x,q_b]\|_2^2 \leq 2 \|x\|_2^2\]
valid for every $x \in \cM$, we can bound the first term by $\sqrt{2 \eta_1}$ and similarly the second term by $\sqrt{2 \eta_2}$. The last term vanishes because $r$ is a PVM. We deduce
\[ \sum_{a,b} \|[p_a,q_b]^2\|_2^2 \leq (\sqrt{2\eta_1}+\sqrt{2\eta_2})^2 \leq 4(\eta_1+\eta_2).\]
The inequality \eqref{eq:commutation_projection} follows because $\eta_1+\eta_2 \leq 4\varepsilon$.

When $A_1=A_{-1}=\{1,-1\}$, we have $p_1 - p_{-1} = 2p_1-1 =1-2p_{-1}$ and similarly for $q$, so that $\|[p_1-p_{-1},q_1-q_{-1}]\|_2 = 4\|[p_a,p_b]\|_2$ for every $a,b \in \{1,-1\}$. Therefore, \eqref{eq:commutation_unitary} is immediate from \eqref{eq:commutation_projection}.
\end{proof}

\subsection{The anticommutation game}
The anticommutation game (or magic square game) is a game with $|X|=15$ with two specific questions $x_1,x_2$ with answers $A(x_1) = A(x_2)=\{-1,1\}$. For later use, we denote \[G_{\anticom}= (X_{\anticom},A_{\anticom},\mu_{\anticom},D_{\anticom})\] the anticommutation game. The specific question will also be denote $x_{\anticom,j}$.

What will be important will not be the precise definition of the game, but that synchronous strategies with large values forces some anti-commutation~:
\begin{lemma}\cite{MagicSquare}\label{lem:magic-square} If $(p_{-1},p_1)$ and $(q_{-1},q_1)$ are the restriction to $x_1$ and $x_2$ of a synchronous strategy for the anticommutation game with value $1-\varepsilon$, then
\[ \|(p_1-p_{-1})(q_1-q_{-1})+ (q_1-q_{-1})(p_1-p_{-1})\|_2^2 \leq 432 \varepsilon.\]
  \end{lemma}
We recall the proof for completeness. To prove the lemma, we need to give the definition of the game. Its set of questions is $X = C \cup L$ where $C=\{1,2,3\}^2$ is a $3\times 3$ square and $L$ is the set of all horizontal or vertical lines in the square. The specific points are $x_1=(1,1)$ and $x_2 = (2,2)$. Define $\alpha(\ell)=1$ for every line except the last vertical line, for which $\alpha(\ell)=-1$. For $c\in C$, define $A(c) = \{-1,1\}$, and for a line $\ell$, define $A(\ell) \subset \prod_{c \in \ell} \{-1,1\}$ by
\[A(\ell) = \{(b_c)_c \in   \prod_{c\in \ell} \{-1,1\} \mid \prod_{c\in \ell} b_c=\alpha(\ell)\}.
\]
The distribution $\mu$ is the uniform distribution on $\{(c,\ell) \mid c \in \ell\}$. And $D(c,\ell,a,b) =1_{a=b_c}$.
\begin{proof}[Proof of Lemma~\ref{lem:magic-square}]
  Let $(p^c_{-1},p^c_1)_{c \in C}$ and $(p^\ell_b)_{b \in A(\ell)}$ be the synchronous strategy with value $\geq 1-\varepsilon$, so that $(p_{-1},p_1) =(p^{1,1}_{-1},p^{1,1}_1)$ and $(q_{-1},q_1) =(p^{2,2}_{-1},p^{2,2}_1)$. For every $c \in C$, define $U^{c} = p^{c}_{-1}-p^c_1$. For every $\ell \in L$ and every  $c\in \ell$, define $U^\ell(c) = \sum_{b \in A(\ell)} b_c p^\ell_b$. Define
  \[ \eta_{c,\ell} = \|U^c - U^\ell(c)\|_2 = 2\sqrt{ 1-\sum_{b \in A(\ell)} \tau(p^c_{b_c} p^\ell_b)}.\]
  So by the assumption that the strategy has value $\geq 1-\varepsilon$, we obtain
  \[ \frac{1}{18} \sum_{\ell} \sum_{c \in \ell} \eta_{c,\ell}^2 =\int \eta_{c,\ell}^2 d\mu(c,\ell) \leq 4\varepsilon.\]
Therefore, if for $\ell \in L$ we denote $\eta_\ell = (\sum_{c \in \ell} \eta_{c,\ell}^2)^{\frac 1 2}$, we obtain
  \[ \sum_\ell \eta_\ell^2 \leq 24 \varepsilon.\]
  The $U^c$ and $U^\ell_c$ are all self-adjoint unitaries. Observe that if $c,c',c''$ are the points in $\ell$, then $U^\ell(c) = \alpha(\ell) U^{\ell}(c') U^{\ell}(c'')$. Therefore, we obtain
  \[ \|U^c - \alpha(\ell) U^{c'} U^{c''}\|_2 \leq \eta_{c,\ell} + \eta_{c',\ell} +\eta_{c'',\ell} \leq \sqrt{3}\eta_\ell.\]
  In the following, we denote by $hi$ the $i$-th horizontal line, and $vj$ the $j$-th vertical line. In the following, we write $M\simeq_\delta N$ if $\|M-N\|_2 \leq \sqrt{3}\delta$. We therefore have
  \begin{align*} U^{11}U^{22} &\simeq_{\eta_{h1}+\eta_{v2}} U^{13}U^{12} U^{12}U^{32} = U^{13}U^{32}\\
    & \simeq_{\eta_{v3}+\eta_{h3}} -U^{23}U^{33}U^{33}U^{31} = -U^{23}U^{31}\\
    & \simeq_{\eta_{h2}+\eta_{v1}} - U^{22}U^{21} U^{21}U^{11} = - U^{22}U^{11}.
  \end{align*}
So we deduce
\[ \|U^{11}U^{22} + U^{22}U^{11}\|_2 \leq \sum_\ell \sqrt{3} \eta_\ell \leq \sqrt{18 \sum_{\ell} \eta_\ell^2} .\]
The lemma follows, because we have already justified that $\sum_\ell \eta_\ell^2 \leq 24 \varepsilon$, and $18 \cdot 24 = 432$.
\end{proof}
\subsection{Pauli matrices}\label{subsection:Pauli}
Let $A$ be a finite group of exponent $2$, that is a group isomorphic to $(\Z/2\Z)^N$ for some integer $N$. In the proof of Corollary~\ref{thm:almost_anticommutation}, we considered two unitary representations $a \in A \mapsto \lambda(a)$ and $\chi \in \hat A \mapsto M(\chi)$ on $B(\ell_2(A))$, called the Pauli representations.  By Fourier transform (Lemma~\ref{lem:Fourier_abelian_group}), these representations correspond to PVMs $(\tau^X_\chi)_{\chi \in \hat A}$ and $(\tau^Z_a)_{a \in A}$, that we will call the Pauli PVMs. If (by fixing a basis) we choose an isomorphism between $A$ and $(\Z/2\Z)^N$ and identify accordingly $\hat A$ with $(\Z/2\Z)^N$ for the duality $\langle a ,b\rangle = (-1)^{\sum_{i=1}^N a_i b_i}$ and $\ell_2(A)$ with $\otimes_{i=1}^N \C^2$, then we have
\[ \tau^X_a = \otimes_{i=1}^N \tau^X_{a_i} \textrm{ where } \tau^X_{0} = \begin{pmatrix} 1/2&1/2\\1/2&1/2
\end{pmatrix}, \tau^X_{1} = \begin{pmatrix} 1/2&-1/2\\-1/2&1/2
\end{pmatrix}\]
and
\[ \tau^Z_a = \otimes_{i=1}^N \tau^Z_{a_i} \textrm{ where } \tau^Z_{0} = \begin{pmatrix} 1&0\\0&0
\end{pmatrix}, \tau^X_{1} = \begin{pmatrix} 0&0\\0&1
\end{pmatrix}.\]

We have the following classical fact.
\begin{lemma}\label{lem:Pauli_matrices} Let $U:A\to \cU(\cM)$ and $V:\hat A\to\cU(\cM)$ be two unitary representations with values in a tracial von Neumann algebra satisfying $U(a) V(\chi) = \chi(a) V(\chi) U(a)$ for every $a\in A,\chi \in\hat A$. Then there is a tracial von Neumann algebra $(\cN,\tau')$ such that $(\cM ,\tau) = (B(\ell_2(A))\otimes \cN,\tr\otimes \tau')$,
  \[ \E_a \overline{\chi(a)}U(a) =\tau^X_\chi \otimes 1_\cN \textrm{ for all }\chi \in\hat A\]
  and
  \[\E_\chi \overline{\chi(a)} V(\chi) = \tau^Z_a\otimes 1_\cN \textrm{ for all }a \in A.\]
  \end{lemma}
\begin{proof} This is a result about the representation theory of the Weyl-Heisenberg group introduced in Corollary~\ref{cor:averageGowersHatami_central_ext_products}, for $B=\widehat A$ and the map $\gamma(a,\chi)=\chi(a)$. Indeed, it is well-known that its irreducible representations are of two kinds~: those that are trivial on the center and one-dimensional (there are $|A|^2$ of them, corresponding to characters of the abelian group $A\times \hat A$), and a unique representation $\pi_0$ that is non trivial on the center, of dimension $|A|$, given by $(\lambda,M)$. Therefore, if $U$ and $V$ are in the lemma, they give rise to a unitary representation $\pi$ of the Weyl-Heisenberg group such that $\pi(Z)=-1_\cM$ (where $Z$ is the non-trivial central element), so by Peter-Weyl $\pi$ is of the form $\pi_0\otimes 1$, and the lemma follows.
\end{proof}

\subsection{Combining the two games}
We now explain how, adapting the construction from \cite{MR3678246}, we can combine the commutation and anticommutation games to obtain the desired game.

Let $H$ be a finite abelian group of exponent $2$ (that is a group isomorphic to $(\Z/2\Z)^N$ for some integer $N$). Let $(\Omega,\prob)$ be a (finite) probability space with two independent random variable $\alpha:\Omega \to H$, $\beta:\Omega \to \hat H$. Define a partition $\Omega= \Omega_+ \cup \Omega_-$ by
\[\Omega_+ = \{\omega \in \Omega \mid \langle \beta(\omega),\alpha(\omega)\rangle = 1\}\]
and
\[\Omega_- = \{\omega \in \Omega \mid \langle \beta(\omega),\alpha(\omega)\rangle = -1.\}\]

This data allows us to define a game $(\cX,\mu,A,D)$ as follows, inspired by the Pauli basis test in \cite{MIPRE}.
\begin{itemize}
\item $\cX = \{PX,PZ\}\cup (X_{com} \times \Omega_+) \cup (X_{\anticom} \times \Omega_-)$
\item $A(PX) = \hat H$, $A(PZ) = H$, for $x \in X_{com}$, $A(x,\omega) = A_{com}(x)$, and for $x \in X_{\anticom}$, $A(x,\omega)=A_{\anticom}(x)$.
\item $\mu$ is the law of $(x,y) \in \cX$ generated as follows~: generate independently $i$ uniformly in $\{1,2,3\}$, $\omega \in \Omega$, $(x_c,y_c)$ according to $\mu_{com}$ and $(x_a,y_a)$ according to $\mu_{\anticom}$. Define
  \[ (x_0,y_0) =\begin{cases} (x_c,y_c) &\textrm{if }\omega \in \Omega_+\\
  (x_a,y_a)&\textrm{otherwise (if }\omega \in \Omega_-).
  \end{cases}\]
  Define $(x,y)$ as
  \begin{equation}\label{eq:def_mu} (x,y) =\begin{cases} (PX,(x_{com,1},\omega)) &\textrm{if $i=1$ and $\omega \in \Omega_+$}\\
  (PX,(x_{\anticom,1},\omega)) &\textrm{if $i=1$ and $\omega \in \Omega_-$}\\
  ((x_0,\omega),(y_0,\omega))&\textrm{if $i=2$}\\
  (PZ,(x_{com,2},\omega)) &\textrm{if $i=3$ and $\omega \in \Omega_+$}\\
  (PZ,(x_{\anticom,2},\omega)) &\textrm{if $i=3$ and $\omega \in \Omega_-$}\\
  \end{cases}.\end{equation}
\item The decision function $D$ is given by: if $x=x_{com,1}$ and $\omega \in \Omega_+$, or if $x=x_{\anticom,1}$ if $\omega \in \Omega_-$,
  \[D(PX,(x,\omega),\chi,\varepsilon) = 1_{\langle \chi,\alpha(\omega)\rangle = \varepsilon}.\]If $x=x_{com,2}$ and $\omega \in \Omega_+$ or $x=x_{\anticom,2}$ and $\omega \in \Omega_-$,
  \[D(PZ,(x,\omega),h,\varepsilon) = 1_{\langle \beta(\omega),h\rangle = \varepsilon},\]
  and in the remaining cases of the support of $\mu$,
  \[D((x,\omega),(y,\omega),a,b) = \begin{cases} D_{com}(x,y,a,b)&\textrm{if }\omega \in \Omega_+\\
  D_{\anticom}(x,y,a,b)&\textrm{if }\omega \in \Omega_-.
\end{cases}.\]
\end{itemize}

\begin{proposition}\label{prop:construction_of_game} Assume that
  \[ \forall 1\neq \chi \in \hat H, \E_\omega \langle \chi, \alpha(\omega)\rangle \leq 1-\frac 1 c\]
  and
  \[ \forall 1\neq h \in H, \E_\omega \langle \beta(\omega),h\rangle \leq 1-\frac 1 {c'}.\]
  
  If the previous game achieves a value $\geq 1-\varepsilon$ on a synchronous strategy, then its restriction $(p^{PX}_\chi)_{\chi \in \hat H}$ and $(p^{PZ}_h)_{h \in H}$ satisfies
  \[ \frac{1}{|H|^2} \sum_{h\in H,\chi \in \hat H}\|U^{PX}(h) U^{PZ}(\chi) - \chi(h) U^{PZ}(\chi) U^{PX}(h)\|_2^2 \leq 1320 cc' \varepsilon,
    \]
  where $U^{PX}$ is the unitary representation of $H$ corresponding to the PVM $p^{PX}$, and $V^{PZ}$ is the unitary representation of $\hat H$ corresponding to the PVM $p^{PZ}$.
\end{proposition}
\begin{proof}
  Denote by $1-\varepsilon_i$ the value of this strategy when the questions asked are conditionned to case $i$ in \eqref{eq:def_mu}. By definition, we then have
  \begin{equation}\label{eq:sumepsiloni} \varepsilon_1+\varepsilon_2+\varepsilon_3\leq 3 \varepsilon.
  \end{equation}
  For every $\omega$, and $j=\{1,2\}$ let $(p^{j,\omega}_{-1},p^{j,\omega}_{1})$ we the PVM corresponding to the question $(x_{com,j},\omega)$ if $\omega \in \Omega_+$ and  to $(x_{\anticom,j},\omega)$ otherwise. These are restrictions to $\{x_1,x_2\}$ of a strategy for the commutation or anticommutation game (depending on whether $\omega \in \Omega_+$ or $\Omega_-$) with value $1-\varepsilon_2(\omega)$, where $\E_\omega \varepsilon_2(\omega) = \varepsilon_2$. So, if we define
  \[U^{\omega} = p^{1,\omega}_{1}-p^{1,\omega}_{-1},\ V^{\omega} = p^{2,\omega}_{1}-p^{2,\omega}_{-1},\] we know from the properties of the commutation and anticommutation game, that
  \[ \E_\omega \|U^{\omega} V^{\omega} - \langle \beta(\omega), \alpha(\omega)\rangle V^{\omega} U^{\omega}\|_2^2 \leq 432 \varepsilon_2.\]
  (and in fact, $432$ can be replaced by $64$ on $\Omega_+$, but this is of no use for us).

  Now by definition of $\varepsilon_1$, we have
  \[ \E_\omega \sum_{\chi \in \hat H} \tau(p^{PX}_\chi p^{1,\omega}_{\langle \chi,\alpha(\omega)}) = \varepsilon_1,\]
  or equivalently
  \[ \E_\omega \| U^{PX}(\alpha(\omega)) - U^\omega\|_2^2 = 4\varepsilon_1.\]
  In the same way,
  \[ \E_\omega \| V^{PZ}(\beta(\omega)) - V^\omega\|_2^2 = 4\varepsilon_3.\]

  Putting everything together, we obtain
  \begin{multline*} \E_\omega \|U^{PX}(\alpha(\omega)) V^{PZ}(\beta(\omega)) - \langle \beta(\omega),\alpha(\omega)\rangle V^{PZ}(\beta(\omega)) U^{PX}(\alpha(\omega))\|_2^2\\ \leq (\sqrt{4\varepsilon_1} + \sqrt{432 \varepsilon_2}+\sqrt{4\varepsilon_3})^2)^2.
  \end{multline*}
  This is less than $440(\varepsilon_1+\varepsilon_2+\varepsilon_3) \leq 1320\varepsilon$ by Cauchy-Schwarz and \eqref{eq:sumepsiloni}. If $\mu$ is the law of $\alpha(\omega)$ and $\nu$ is the law of $\beta(\omega)$, by the assumption that $\alpha$ and $\beta$ are independent, we can write this as
\[ \int \|U^{PX}(h) V^{PZ}(\chi) - \langle \chi,h\rangle V^{PZ}(\chi) U^{PX}(h)\|_2^2 d\mu(\chi) d\nu(h) \leq 1320 \varepsilon.\]
We conclude by Corollary~\ref{thm:almost_anticommutation}.
\end{proof}
The next corollary is expressed in terms of the Pauli PVMs introduced in subsection \ref{subsection:Pauli}.
\begin{corollary}\label{cor:pauli_from_good_strategy}
 Under the same assumption on $\alpha,\beta$ as in Proposition~\ref{prop:construction_of_game}, then any synchronous strategy with value $1-\varepsilon$ to $G_N$ is $O(cc'\varepsilon)$-close to a strategy on an algebra of the form $(M_{2^N}(\C)\otimes \cN,\mathrm{\tr}\otimes \tau')$ where $P^{PX}_\chi = \tau^X_\chi\otimes 1_\cN$ and $P^{PZ}_h = \tau^Z_h\otimes 1_\cN$ for all $\chi \in\hat H, h\in H$.
\end{corollary}
\begin{proof}By the conclusion of Proposition~\ref{prop:construction_of_game} and Corollary~\ref{thm:almost_anticommutation_2group}, the representations $U^{PX},U^{PZ}$ are $O(cc'\varepsilon)$-close to a pair of representations $U_1,V_1$ satisfying  \[ U_1(h) V_1(\chi) = \chi(h) V_1(\chi) U_1(h)\forall h \in H,\chi \in \hat H.\]
  By Lemma~\ref{lem:Pauli_matrices} $U_1$ and $V_1$ have the desired forms, and by Lemma~\ref{lem:close_strategies_unitaries} the closeness of the unitary representations is equivalent to the closeness of the strategies.
  \end{proof}
\begin{example}\label{ex:game_from_code} Let $C,C'$ be two linear binary codes of the same dimension $N$, say with parameters $[k,N,d]$ and $[k',N,d']$. By Example~\ref{ex:code_spectral_gap}, any choice of a basis for $C$ gives rise to a probability measure on $(\Z/2\Z)^N$ that is uniform on a subset of cardinality $k$, and its spectral gap constant is $\kappa = k/2d$. Similarly, any choice of a basis for $C'$ produces a probability measure on $(\Z/2\Z)^N$ uniform on a subset of size $k'$ and with $\kappa = k'/2d'$. 

  Let us consider $\Omega = \mathrm{supp}(\mu)\times  \mathrm{supp}(\mu')$ with its uniform probability measure, and $\alpha,\beta:\Omega\to (\Z/2\Z)^N$ the two coordinate projections. If we identify $(\Z/2\Z)^N$ with its Pontryagin dual for the duality $\langle a,b\rangle=(-1)^{\sum_i a_i b_i}$, the previous construction therefore gives rise to a two-player non-local game $G(C)=(X_{C,C'},\mu,A_{C,C'},D_{C,C'})$ with $|X_{C,C'}|=O(kk')$, $A_{C,C'}=(\Z/2\Z)^N$  and two particular questions $PX$ and $PZ$ with the following properties~
  \begin{itemize}
  \item $\mu(PX)=\mu(PZ)=\frac 1 3$,
  \item Any synchronous strategy with value $1-\varepsilon$ to $G(C)$ is $O(\varepsilon\frac{kk'}{dd'})$-close to a strategy on an algebra of the form $(M_{2^N}(\C)\otimes \cN,\mathrm{\tr}\otimes \tau')$ where $P^{PX}_a = \tau^X_a\otimes 1_\cN$ and $P^{PZ}_b = \tau^Z_b\otimes 1_\cN$ for all $a,b\in (\Z/2\Z)^N$.
  \end{itemize}
\end{example}
The existence of asymptotically good codes (Proposition~\ref{prop:measure_with_small_support_and_small_Fourier_transform_nonabelian}) implies in particular the following Theorem.
\begin{theorem}\label{thm:mainTheorem}
  For every $N$, there is a game $G_N$ with $|X|\leq C N^2$, $|A|=2^N$ and with two specific questions $PX,PZ \in X$ satisfying $\mu(PX)=\mu(PZ) = \frac 1 3$, with answer sets $A(PX)=(\Z/2\Z)^N = A(PZ)=(\Z/2\Z)^N$ such that any synchronous strategy with value $1-\varepsilon$ to $G_N$ is $O(\varepsilon)$-close to a strategy on an algebra of the form $(M_{2^N}(\C)\otimes \cN,\mathrm{\tr}\otimes \tau')$ where $P^{PX}_a = \tau^X_a\otimes 1_\cN$ and $P^{PZ}_b = \tau^Z_b\otimes 1_\cN$ for all $a,b\in (\Z/2\Z)^N$.
\end{theorem}
Moreover, the $G_N$ can be made \emph{explicit}, replacing the existencial argument of asymptotically good codes from Proposition~\ref{prop:measure_with_small_support_and_small_Fourier_transform_nonabelian} by explicit constructions, for examples such as Justesen codes \cite{MR0465509} or expander codes \cite{MR2494807}, see also \cite{MR1996953}.
\begin{proof} Combine Proposition~\ref{prop:measure_with_small_support_and_small_Fourier_transform_nonabelian} and Corollary~\ref{cor:pauli_from_good_strategy}.
\end{proof}
\subsection{Final comment}\label{subsection:comparison}
Let $q=2^k$ be a power of $2$ with $k$ odd. This guarantees that the Pontryagin dual of $\F_q$ identifies with $\F_q$ for the duality bracket $\langle x,y\rangle = (-1)^{Tr(xy)}$ where we identify an element of $\F_q$ (here $xy$) with the $\F_2$-linear map of multiplication on $\F_q$, seen as an $\F_2$ vector space, and so $Tr(xy) \in \F_2$ is the trace of this $\F_2$-linear operator. Identify the Pontryagin dual of $\F_q^N$ with itself accordingly.

Let $m$ be an integer. Consider the Reed-Muller code $C$, the set of all polynomials of individual degree $\leq 1$ in $m$ variables, seen as a subspace of the space $\F_q^{\F_q^m}$ of all functions $\F_q^m\to \F_q$. By the Schwarz-Zippel Lemma, any nonzero such polynomial has at most $mq^{m-1}$ zeros, so it is a $[q^m,2^m,\leq q^m(1-m/q)]_q$-code. By Example~\ref{ex:code_spectral_gap}, this code therefore gives rise to a probability measure $\mu$ on $\F_q^{2^m}$ that is uniformly supported on a set of size $q^{m+1}$, and such that $\kappa(\mu) \leq \frac{q-1}{q-m}$. In particular, as soon as $q \geq 2m$, we obtain $\kappa(\mu) \leq 2$.

If we define a game from two copies of $C$ as in Example~\ref{ex:game_from_code}, we therefore obtain a game with $|X| = O(2^{2k(m+1)})$, $|A| = 2^{k2^m}$ and that satisfies the same conclusion as in Theorem~\ref{thm:mainTheorem}. In this example, the dependance between the number of questions $O(2^{2k(m+1)})$ and of answers $O(2^{k 2^m})$ is not as good as in Theorem~\ref{thm:mainTheorem}.

In \cite[Section 7.3]{MIPRE} a game called the Pauli basis test is
studied, depending on the same parameters $k$ and $m$ and an
additional parameter $d$. We will not recall the precise description
of the Pauli basis test here, but it is closely related to the game
previously defined for this value of $H,\alpha,\beta$. In particular,
it is not difficult to show that, for every $d\geq 1$, the Pauli basis
test contains the game above, and therefore a conclusion similar to
Theorem~\ref{thm:mainTheorem} holds for the Pauli basis test as soon
as $q \geq 2m$. This is significantly better than in \cite[Theorem
  7.14]{MIPRE}, where this statement is proved for the Pauli basis
test, but with $O(\varepsilon)$ replaced by $a(md)^a(\varepsilon^b +
q^{-b}+2^{-bmd})$ for some constants $a,b,c$.

Therefore, Theorem~\ref{thm:mainTheorem} and the more general construction in Example~\ref{ex:game_from_code} can be seen as both an improvement, generalization and simplification of   \cite[Theorem 7.14]{MIPRE}.

However, it should be noted that this simplification does not remove
all the dependances to the difficult result from
\cite{MR4399717,MR4518868}, as this result is still used in the answer
reduction (or PCP) part of \cite{MIPRE}. The main difference between
the results in this note and \cite{MR4518868} is that in
\cite{MR4518868}, the number of answers is much smaller.

\section{Stability}\label{sec:stability}

\begin{definition} Let $G$ be a countable group and $\mu$ a probability measure with generating support. A map $\varphi:G\to \cU(\cM)$ is said to be an $(\varepsilon,\mu)$-almost homomorphism if $\iint \| \varphi(gh) - \varphi(g) \varphi(h)\|_2^2d\mu(g) d\mu(h) \leq \varepsilon$.
\end{definition}
Let $\vn$ be a class of von Neumann algebras equipped with normal tracial states. 

\begin{definition}\label{def:stability}
  We say that $(G,\mu)$ is $\vn$-stable if there is a non-decreasing
  function $\delta:[0,4]\to [0,4]$ with $\lim_{t \to 0} \delta(t)= 0$ and such that, for every $(\cM,\tau)$ in $\vn$ and every $(\varepsilon,\mu)$-almost homomorphism $\varphi: G \to \cU(\cM)$, there is a group homomorphism $\pi:G \to
  \cU(\cM)$ satisfying
  \[ \int \|\varphi(g) - \pi(g)\|_2^2 d\mu(g) \leq \delta(\varepsilon).\]

Such a function $\delta$ satisfying moreover that $\delta$ is concave and $\delta(4)= 4$ will be called a modulus of $\vn$-stability.
\end{definition}
\begin{remark} If $(G,\mu)$ is $\vn$-stable, then it admits a modulus of $\vn$-stability (by replacing $\delta$ by the smallest function greater than $\delta$, concave and taking the value $4$ at $4$). The concavity requirement for $\delta$ is here to make statements such as Theorem~\ref{thm:stability_direct_product} cleaner. It is also very natural, and in fact if $\vn$ is stable by direct sums then the best $\delta$ is necessarily concave. Indeed, if $\varphi_1$ and $\varphi_2$ are respectively $(\varepsilon_1,\mu)$ and $(\varepsilon_2,\mu)$ almost representations with values in $\cU(\cM_1)$ and $\cU(\cM_2)$, and $\lambda \in [0,1]$, then defining $\cM = \cM_1\oplus \cM_2$ with trace $\tau(x_1,x_2) = \lambda \tau_1(x_1) + (1-\lambda)\tau_2(x_2)$, the pair $(\varphi_1,\varphi_2)$ defines a $(\lambda\varepsilon_1+(1-\lambda)\varepsilon_2,\mu)$ almost representation $\varphi$ on $\cM$. Moreover, a unitary representation $\pi:G\to \cU(\cM)$ is the same as a pair of unitary representations $\pi_i:G\to \cU(\cM_i)$, with
  \begin{multline*} \int \|\varphi(g) - \pi(g)\|_2^2 d\mu(g) \\=\lambda \int \|\varphi_1(g) - \pi_1(g)\|_2^2 d\mu(g) + (1-\lambda) \int \|\varphi_2(g) - \pi_2(g)\|_2^2 d\mu(g).
  \end{multline*}
  Taking the supremum over all $(\varepsilon_1,\mu)$ and $(\varepsilon_2,\mu)$ almost representations, we obtain
  \[ \lambda\delta(\varepsilon_1) + (1-\lambda)\delta(\varepsilon_2) \leq \delta(\lambda\varepsilon_1+(1-\lambda)\varepsilon_2).\]
  That is, $\delta$ is concave.
\end{remark}

\begin{definition}\label{def:flexible_stability}
  We say that $(G,\mu)$ is $\vn$ flexibly stable if there is a non-decreasing function $\delta:[0,4]\to [0,4]$ such that $\lim_{t \to 0} \delta(t) = 0$ such that, for every $(\cM,\tau)$ in $\vn$ and every   $(\varepsilon,\mu)$-almost homomorphism $\varphi: G \to \cU(\cM)$, there is a projection $P \in \cM_\infty$, a group homomorphism $\pi:G \to \cU(\cM)$ and an isometry $w \in P \cM_\infty 1_\cM$ satisfying
  \[\max\left(\tau(P)-1, \int \|\varphi(g) - w^*\pi(g)w\|_2^2 d\mu(g)\right) \leq \delta(\varepsilon).\]

  Such a function $\delta$ satisfying moreover that $\delta$ is concave and $\delta(4)=4$ will be called a modulus of $\vn$-flexible stability.
\end{definition}
For example, Theorem~\ref{thm:averageGowersHatami} says that for every finite group and any $\vn$, $(G,\prob_G)$ is $\vn$ flexibly stable with modulus $\delta(t) = \min(4,169 t)$.

This is an adaptation of the standard notions of Hilbert-Schmidt stability and Hilbert-Schmidt flexible stability, see for example \cite{Ioana2}.

Let $\vn_{\mathrm{fin}}$ denote the class of all finite-dimensional von Neumann algebras with tracial states.
\begin{lemma} A countable group $G$ is Hilbert-Schmidt (flexibly) stable if and only if there is a probability measure $\mu$ on $G$ with generating support such that $(G,\mu)$ is $\vn_{\mathrm{fin}}$ (flexibly) stable.

  If $G$ is finitely presented, $\mu$ can moreover be taken to be of finite support.
\end{lemma}
\begin{proof}
  The lemma is immediate if $\vn_{\mathrm{fin}}$ is replaced by the set \[\{M_n\} = \{(M_n(\C),\tr) \mid n\geq 1.\}.\] So the whole point of the lemma is to show that if $(G,\mu)$ is $\{M_n\}$-stable, then $(G,\mu)$ is $\vn_{\mathrm{fin}}$-stable, and similarly for flexible stability.

  Assume that $(G,\mu)$ is $\{M_n\}$-stable, and let $\delta$ be a modulus. We shall prove that $(G,\mu)$ is $\vn_{\mathrm{fin}}$-stable with the same modulus. Let $(\cM,\tau) \in \vn_{\mathrm{fin}}$ and $\varphi:G\to \cU(\cM)$ be an $(\varepsilon,\mu)$-almost homomorphism. We can decompose $\cM$ into a finite direct sum of matrix algebras $\cM = \oplus_{i=1}^k M_{n_i}(\C)$ with trace $\tau = \sum_i \lambda_i \tr_{n_i}$ for positive numbers $\lambda_i$ summing to $1$. Then $\varphi(g) = (\varphi_i(g))_i$ where $\varphi_i$ is an $(\varepsilon_i,\mu)$-almost homomorphism for some numbers $\varepsilon_i$ satisfying
  \[\sum\nolimits_i \lambda_i\varepsilon_i\leq\varepsilon.\]
  Therefore, by $\{M_n\}$-stability, there is homomorphism $\pi_i \colon G\to\cU(n_i)$ such that $\int\|\varphi_i(g) - \pi_i(g)\|_2^2d\mu(g) \leq \delta(\varepsilon_i)$. If we define a homomorphism $\pi:G\to \cU(\cM)$ by $\pi(g) = (\pi_i(g))_{g\in G}$, we obtain
  \[ \int \|\varphi(g) - \pi(g)\|_2^2 d\mu(g) = \sum_{i=1}^k\lambda_i \int\|\varphi_i(g) - \pi_i(g)\|_2^2d\mu(g) \leq \delta(\varepsilon)\]
  by concavity of $\delta$.

The argument is identical for flexible stability, and left to the reader.
\end{proof}

Assume that $\vn$ is a class of von Neumann algebras closed under taking subalgebras.

\begin{theorem}\label{thm:stability_direct_product} Direct products of $\vn$-stable groups are $\vn$-stable, provided that one of them has property (T).

  More precisely, if $(G_1,\mu_1)$ has property (T) and is $\vn$-stable with modulus $\delta_1$ and $(G_2,\mu_2)$ is $\vn$-stable with modulus $\delta_2$, then $(G_1 \times G_2,\mu)$ is $\vn$-stable with modulus $\delta(\varepsilon) \lesssim  \delta_2(\kappa(\mu_1) \delta_2(\varepsilon))$, where $\mu$ is the probability measure
  \[\mu(x,y) = \frac{1}{2}(\mu(x) 1_{y=1}+\mu(y) 1_{x=1}).\]
\end{theorem}
The proof will use the following simple lemmas.
\begin{lemma}\label{eq:almost-unitary_close_to_unitary} Let $(\cM,\tau)$ be a tracial von Neumann algebra and $\cN\subset\cM$ be a subalgebra, and let $E_\cN:\cM \to \cN$ be the conditional expectation. For every unitary $V \in \cU(\cM)$, there is a unitary $\tilde{V}\in\cU(\cN)$ such that $\|V-\tilde{V}\|_2 \leq \sqrt{2} \|V-E_\cN(V)\|_2$.
\end{lemma}
The $\sqrt{2}$ is optimal, for example when $E_\cN(V)=0$.
\begin{proof} Let $X = E_\cN(V)$. Since $\cN$ is finite, we can write $X=\tilde{V}|X|$ where $\tilde{V} \in\cU(\cN)$ and $|X|=(X^*X)^{\frac 1 2}$. By the duality between $L_1(\cM,\tau)$ and $\cM$, we have $\Re \tau(V^*X) \leq\tau(|X|)$, or equivalently  $\|\tilde{V}-X\|_2 \leq \|V-X\|_2$. If we decompose $V-\tilde{V} = V-X + X-\tilde{V}$, the terms $V-X$ and $X-\tilde{V}$ are orthogonal, and therefore
  \[\|V-\tilde{V}\|_2^2 = \|V-X\|^2 + \|\tilde{V}-X\|^2 \leq 2\|V-X\|_2^2.\]
  This proves the lemma.
\end{proof}
\begin{lemma}\label{eq:norm_of_conditional_exp} Let $\cN\subset\cM$ be a von Neumann sulalgebra, and $E_\cN:\cM\to\cN$ be the trace-preserving conditional expectation. For every $\xi \in L_2(\cM,\tau)$,
  \[ \|\xi - E_\cN(\xi)\|_2 = \sup\{\tau(\xi \eta) \mid \eta \in L_2(\cM,\tau), E_\cN(\eta)=0,\|\eta\|_2 = 1\}.\]
\end{lemma}
\begin{proof}
  $(1-E_\cN)$ is the orthogonal projection on $\{\eta \in L_2(\cM,\tau) \mid E_\cN(\eta)=0\}$.
\end{proof}

\begin{proof}[Proof of Theorem~\ref{thm:stability_direct_product}] 
  Let $\varphi : G_1\times G_2\to\cU(\cM)$ be a $(\varepsilon,\mu)$-almost homomorphism. The idea is simple: we first use the stability of $G_1$ to say that the restriction of $\pi$ to $G_1$ is close to a representation. Then by property (T), we will deduce that the restriction of $\pi$ to $G_2$ is close to an almost homomorphism from $G_2$ to the commutant of $G_1$, and therefore by stability of $G_2$ to an actual homomorphism with values in the commutant of $G_1$.

Here are the details. If we denote
  \[\varepsilon_{i,j} = \iint_{G_i\times G_j} \| \varphi(gh) - \varphi(g) \varphi(h)\|_2^2d\mu_i(g) d\mu_j(h),\]
  we then have $\varepsilon_{1,1}+\varepsilon_{2,2} + \varepsilon_{1,2}+\varepsilon_{2,1}\leq 4\varepsilon$.

  The restriction of $\varphi$ to $G_1$ is an $(\varepsilon_{1,1},\mu_1)$-almost representation of $G_1$, so there is a unitary representation $\pi_1:G_1\to\cU(\cM)$ such that
  \[\int \|\varphi(g)-\pi_1(g)\|_2^2 d\mu_1(g) \leq \delta_1(\varepsilon_{1,1}).\]
  Denote $\cN=\cM \cap \tilde{\pi}_1(G_1)'$, and $E:\cM\to\cN$ the conditional expectation. Since $\vn$ is assumed to be stable under taking subalgebras, we have $\cN\in\vn$. Define, for every $h \in G_2$,
  \[\eta(h) = \|\varphi(h) -E(\varphi(h))\|_2.\]
We first establish an upper bound for the norm of $\eta$ in $L_2(G_2,\mu_2)$. This is where we use that $G_1$ has property (T). By Lemma~\ref{lemma:spectral_gap_commutator}, we know that
  \[ \eta(h)^2 \leq \kappa(\mu_1) \int_{G_1} \|[\varphi(h),\pi_1(g)]\|_2^2 d\mu_1(g).\]
  We can bound
  \[\|[\varphi(h),\pi_1(g)\|_2 \leq 2\|\varphi(g)-\pi_1(g)\|_2 + \|\varphi(h)\varphi(g) -\varphi(hg)\|_2 + \|\varphi(g) \varphi(h) - \varphi(gh)\|_2,\]
 So by the triangle inequality in $L_2(\mu_1\times \mu_2)$, we obtain
  \[ \|\eta\|_{L_2(\mu_2)} \leq \sqrt{\kappa(\mu_1)/2}(2 \sqrt{\delta_1(\varepsilon_{1,1})} + \sqrt{\varepsilon_{1,2}} + \sqrt{\varepsilon_{2,1}}).\]
  By the standing assumptions on $\delta_1$, and the Cauchy-Schwarz inequality, we obtain the following inequality that we will use shortly
  \begin{equation}\label{eq:L2norm_of_eta} \int\eta(h)^2 d\mu_2(h) \leq 12 \kappa(\mu_1) \delta_1(\varepsilon).\end{equation}

By Lemma~\ref{eq:almost-unitary_close_to_unitary}, for every $h \in G_2$ there is a unitary $V(h) \in \cU(\cN)$ such that $\|\varphi(h) - V(h)\|_2\leq \sqrt{2}\eta(h)$. Our goal is to show that $V$ is an almost-homomorphism, to then apply flexible stability for $G_2$. By the triangle inequality, 
  \[ \left(\int \|V(gh)-V(g)V(h)\|_2^2 d\mu_2(g) d\mu_2(h)\right)^{\frac 1 2}\]
  is bounded above by
  \[ \sqrt{\varepsilon_{2,2}} + 2\sqrt{2} \|\eta\|_{L_2(\mu_2)} +  \sqrt{2}\|\eta\|_{L_2(\mu_2\ast\mu_2)}.\]  
  Decompose $\varphi(gh) = a+b+c$ with $a= \varphi(gh)-\varphi(g)\varphi(h)$, $b= \varphi(g) (\varphi(h) - E_\cN(\varphi(h)))$ and $c=\varphi(g) E_\cN(\varphi(h))$. Using Lemma~\ref{eq:norm_of_conditional_exp}, and writing $\sup_\eta$ for the supremum over all unit vectors in the orthogonal of $L_2(\cN,\tau)$ in $L_2(\cM,\tau)$, we have
\begin{align*}
  \eta(gh) & = \sup_\eta \tau(a\eta) + \tau(b\eta)+ \tau(c\eta)\\
  & \leq \|a\|_2 + \|b\|_2 +\|c - E_\cN(c)\|_2\\
  & \leq \|\varphi(gh) - \varphi(g) \varphi(h)\|_2 + \eta(h)+\eta(g).
\end{align*}
  As a consequence, we have $\|\eta\|_{L_2(\mu_2\ast\mu_2)} \leq \sqrt{\varepsilon_{2,2}} + 2\|\eta\|_{L_2(\mu_2)}$, and we deduce
  \[ \left(\int \|V(gh)-V(g)V(h)\|_2^2 d\mu_2(g) d\mu_2(h)\right)^{\frac 1 2} \leq (1+\sqrt{2})\sqrt{\varepsilon_{2,2}} + 4\sqrt{2} \|\eta\|_{L_2(\mu_2)}.\]
By \eqref{eq:L2norm_of_eta}, this last quantity is less than $\sqrt{C\kappa(\mu_1) \delta_1(\varepsilon)}$ for some universal constant $C$. The theorem follows easily. Indeed, by stability for $G_2$, we deduce that there is a unitary representation $\pi_2:G\to\cU(\cN)$ such that
\[ \int \|V(g) - \pi_2(g)\|_2^2d\mu_2(g) \leq \delta_2(C\kappa(\mu_1)\delta_1(\varepsilon)),\]
and therefore 
\[ \int \|\varphi(g)- \pi_2(g)\|_2^2d\mu_2(g) \leq 3(\|\eta\|_{L_2(\mu_2)}^2 + \delta_2(C\kappa(\mu_1)\delta_1(\varepsilon))). \]
By definition of $\cN$, $\pi_2(g)$ commutes with $\pi_1(G_1)$, so the pair $({\pi}_1,{\pi}_2)$ gives rise to a unitary representation ${\pi} : G_1\times G_2\to \cU(\cM)$ by ${\pi}(g,h)={\pi}_1(g){\pi}_2(h)$. We have
\begin{align*} \int \|\varphi(g) - {\pi}(g)\|_2^2 d\mu(g)
  &\leq \frac{1}{2}(\delta_1(\varepsilon_{1,1}) + 3\|\eta\|_{L_2(\mu_2)}^2 + 3\delta_2(C\kappa(\mu_1)\delta_1(\varepsilon))).
\end{align*}
This is $\lesssim \delta_2(\kappa(\mu_1)\delta_1(\varepsilon))$ by \eqref{eq:L2norm_of_eta} and the standing assumption that the moduli $\delta_i$ are concave and satisfy $\delta_i(t) \geq t$.
\end{proof}
\begin{theorem}\label{thm:flexible_stability_direct_product} Direct products of $\vn$-flexibly stable groups are $\vn$-flexibly stable, provided that one of them has property (T).

  More precisely, if $(G_1,\mu_1)$ has property (T) and is $\vn$-flexibly stable with modulus $\delta_1$ and $(G_2,\mu_2)$ is $\vn$-flexibly stable with modulus $\delta_2$, then $(G_1 \times G_2,\frac{1}{2}(\mu_1 \otimes \delta_1 + \delta_1 \otimes \mu_2))$ is $\vn$-flexibly stable with modulus $\delta(\varepsilon) \lesssim  \delta_2(\kappa(\mu_1) \delta_2(\varepsilon))$.
\end{theorem}
\begin{proof}
  This is proved in the same way as for stability. The only difference is that the von Neumann algebras change after each use of flexible stability. First, when flexible stability is used for $G_1$, a representation $\tilde{\pi}_1$ of $G_1$ is constructed in a small dilation $\cM_1$ of $\cM$. The almost representation of $G_2$ is then defined with values in $\cU(\cM_1)$ as $w_1 \pi(g) w_1^* +P_1 - w_1w_1^*$. Similarly, when flexible stability is used for $G_2$ to construct a representation $\tilde{\pi}_2$ in a small dilation $\cM_2$ of $\cM_1$, the representation $\tilde{\pi}_1$ of $G_1$ is then defined with values in $\cU(\cM_2)$ as
$w_2 \tilde{\pi}_1(g) w_2^* +P_2 - w_2w_2^*$.
\end{proof}
One way of expressing the crucial step in the above proofs is as follows~: by (the baby case of) von Neumann's bicommutant theorem, the centralizer of a subgroup of $\cU(n)$ decomposes as a direct product of smaller unitary groups. This is not true for permutation groups, and therefore the proof of the preceding theorems does not apply verbatim for permutation stability or permutation flexible stability. The following question is however natural. A positive answer would pleasingly complement Ioana's results \cite{MR4134896}.
\begin{question} Is it true that the direct product of two (flexibly) permutation stable groups is  (flexibly) permutation stable, provided that one of them has property ($\tau$)?
\end{question}

\bibliographystyle{plain}
\bibliography{biblio}

\def\cprime{$'$}
\begin{thebibliography}{10}

\bibitem{MR1262979}
Noga Alon and Yuval Roichman.
\newblock Random {C}ayley graphs and expanders.
\newblock {\em Random Structures Algorithms}, 5(2):271--284, 1994.

\bibitem{AnantharamanPopa}
Claire Anantharama and Sorin Popa.
\newblock {\em An introduction to {$\mathrm{II}_1$} factors}.
\newblock Book available
  \url{https://www.idpoisson.fr/anantharaman/publications/IIun.pdf}.

\bibitem{beckerChapman}
Oren Becker and Michael Chapman.
\newblock Stability of approximate group actions: uniform and probabilistic.
\newblock {\em J. Eur. Math. Soc. (JEMS)}, 25(9):3599--3632, 2023.

\bibitem{MagicSquare}
A.~Coladangelo and J~Stark.
\newblock Robust self-testing for linear constraint system games.
\newblock {\em arXiv:1709.09267}, 2017.

\bibitem{MR3867328}
Marcus De~Chiffre, Narutaka Ozawa, and Andreas Thom.
\newblock Operator algebraic approach to inverse and stability theorems for
  amenable groups.
\newblock {\em Mathematika}, 65(1):98--118, 2019.

\bibitem{orthonormalisation}
Mikael de~la Salle.
\newblock Orthogonalization of positive operator valued measures.
\newblock {\em C. R., Math., Acad. Sci. Paris}, 360:549--560, 2022.

\bibitem{MR3733361}
W.~T. Gowers and O.~Hatami.
\newblock Inverse and stability theorems for approximate representations of
  finite groups.
\newblock {\em Mat. Sb.}, 208(12):70--106, 2017.

\bibitem{414657}
Jason~Gaitonde (https://mathoverflow.net/users/170770/j g).
\newblock Probability measure on the boolean cube with small support and small
  fourier transform.
\newblock MathOverflow.
\newblock URL:https://mathoverflow.net/q/414657 (version: 2022-01-25).

\bibitem{MR1996953}
W.~Cary Huffman and Vera Pless.
\newblock {\em Fundamentals of error-correcting codes}.
\newblock Cambridge University Press, Cambridge, 2003.

\bibitem{MR4134896}
A.~Ioana.
\newblock Stability for product groups and property ({$\tau$}).
\newblock {\em J. Funct. Anal.}, 279(9):108729, 32, 2020.

\bibitem{Ioana2}
A.~Ioana.
\newblock Almost commuting matrices and stability for product groups.
\newblock {\em J. Eur. Math. Soc. (JEMS)}, 27(10):4027--4068, 2025.

\bibitem{MIPRE}
Z.~Ji, A.~Natarajan, T.~Vidick, J.~Wright, and H.~Yuen.
\newblock {MIP*=RE}.
\newblock {\em arXiv:2001.04383}, 2020.

\bibitem{MR4518868}
Zhengfeng Ji, Anand Natarajan, Thomas Vidick, John Wright, and Henry Yuen.
\newblock Quantum soundness of testing tensor codes.
\newblock {\em Discrete Anal.}, pages Paper No. 17, 73, 2022.

\bibitem{MR4399717}
Zhengfeng Ji, Anand Natarajan, Thomas Vidick, John Wright, and Henry Yuen.
\newblock Quantum soundness of testing tensor codes.
\newblock In {\em 2021 {IEEE} 62nd {A}nnual {S}ymposium on {F}oundations of
  {C}omputer {S}cience---{FOCS} 2021}, pages 586--597. IEEE Computer Soc., Los
  Alamitos, CA, [2022] \copyright 2022.

\bibitem{MR2221038}
Martin Kassabov, Alexander Lubotzky, and Nikolay Nikolov.
\newblock Finite simple groups as expanders.
\newblock {\em Proc. Natl. Acad. Sci. USA}, 103(16):6116--6119, 2006.

\bibitem{MR2097328}
Zeph Landau and Alexander Russell.
\newblock Random {C}ayley graphs are expanders: a simple proof of the
  {A}lon-{R}oichman theorem.
\newblock {\em Electron. J. Combin.}, 11(1):Research Paper 62, 6, 2004.

\bibitem{MR0465509}
F.~J. MacWilliams and N.~J.~A. Sloane.
\newblock {\em The theory of error-correcting codes. {I}}.
\newblock North-Holland Mathematical Library, Vol. 16. North-Holland Publishing
  Co., Amsterdam-New York-Oxford, 1977.

\bibitem{MR3678246}
Anand Natarajan and Thomas Vidick.
\newblock A quantum linearity test for robustly verifying entanglement.
\newblock In {\em S{TOC}'17---{P}roceedings of the 49th {A}nnual {ACM} {SIGACT}
  {S}ymposium on {T}heory of {C}omputing}, pages 1003--1015. ACM, New York,
  2017.

\bibitem{MR2494807}
Tom Richardson and R\"{u}diger Urbanke.
\newblock {\em Modern coding theory}.
\newblock Cambridge University Press, Cambridge, 2008.

\bibitem{MR1873025}
M.~Takesaki.
\newblock {\em Theory of operator algebras. {I}}, volume 124 of {\em
  Encyclopaedia of Mathematical Sciences}.
\newblock Springer-Verlag, Berlin, 2002.
\newblock Reprint of the first (1979) edition, Operator Algebras and
  Non-commutative Geometry, 5.

\bibitem{vidickPauliBraiding}
T.~Vidick.
\newblock Pauli brading.
\newblock {\em available at
  \url{http://users.cms.caltech.edu/~vidick/notes/pauli_braiding_1.pdf} and
  \url{https://mycqstate.wordpress.com/2017/06/28/pauli-braiding/}}, 2017.

\end{thebibliography}

\end{document}